\documentclass[12pt,oneside]{amsart}
\usepackage[margin=1in]{geometry}

\usepackage{graphicx}
\usepackage{color}
\usepackage{amssymb}
\usepackage{tikz}
\usepackage[normalem]{ulem}
\usepackage{hyperref} 

\usepackage{amsmath,amsthm,amscd,amssymb}
\usepackage{latexsym}

\newcommand{\bF}{\mathbb{F}}

\newcommand{\cE}{\mathcal{E}}

\newcommand{\cH}{\mathcal{H}}

\newcommand{\e}{\varepsilon}
\newcommand{\es}{\varnothing}
\newcommand{\sm}{\setminus}
\newcommand{\n}{\noindent}
\renewcommand{\v}{\vspace{.2cm}}

\renewcommand{\hat}{\widehat}

\usepackage{ulem}
\usepackage{soul}
\numberwithin{equation}{section}
\theoremstyle{plain}
\newtheorem{theorem}{Theorem}[section]
\newtheorem{lemma}[theorem]{Lemma}
\newtheorem{corollary}[theorem]{Corollary}
\newtheorem{proposition}[theorem]{Proposition}

\theoremstyle{definition}
\newtheorem{definition}[theorem]{Definition}

\theoremstyle{remark}
\newtheorem{remark}[theorem]{Remark}

\newtheorem{case[theorem]}{Case}

\def\R{\mathbb R}

\date{November 11, 2025}      

\begin{document} 
\title{VC-Dimension of Salem sets over finite fields} 

\author{Moustapha Diallo} 
\address{Department of Mathematics \\ University of Georgia}
\email{moustapha.diallo@uga.edu} 

\author{Brian McDonald} 
\address{Department of Mathematics \\ University of Georgia} 
\email{brian.mcdonald@uga.edu}

\thanks{}


\begin{abstract}
\n The VC-dimension, introduced by Vapnik and Chervonenkis in 1968 in the context of learning theory, has in recent years provided a rich source of problems in combinatorial geometry.  Given $E\subseteq \mathbb{F}_q^d$ or $E\subseteq \mathbb{R}^d$, finding lower bounds on the VC-dimension of hypothesis classes defined by geometric objects such as spheres and hyperplanes is equivalent to constructing appropriate geometric configurations in $E$.  The complexity of these configurations increases exponentially with the VC-dimension.
\v

\n These questions are related to the Erd\H{o}s distance problem and the Falconer problem when considering a hypothesis class defined by spheres.  In particular, the Erd\H{o}s distance problem over finite fields is equivalent to showing that the VC-dimension of translates of a sphere of radius $t$ is at least one for all nonzero $t\in \mathbb{F}_q$.  In this paper, we show that many of the existing techniques for distance problems over finite fields can be extended to a much broader context, not relying on the specific geometry of circles and spheres.  We provide a unified framework which allows us to simultaneously study highly structured sets such as algebraic curves, as well as random sets.

\end{abstract}

\maketitle

\tableofcontents

\section{Introduction}

\subsection{Distance problems}

\n In this paper, we study various generalizations of the Erd\H{o}s distance problem over finite fields.  For a set of $n$ points $E\subseteq \mathbb{R}^2$, Erd\H{o}s investigated the size of the distance set
$$
\Delta(E):=\{|x-y|: x,y\in E\}.
$$
The size of the distance set is trivially bounded above,
$$
|\Delta(E)|\leq \binom{n}{2}.
$$
Here and throughout, we denote the cardinality of a finite set $S$ by $|S|$.  This trivial upper bound is sharp, since it is not hard to find examples of sets $E\subseteq \mathbb{R}^2$ of arbitrary size such that all distances are distinct. Lower bounds are much more interesting.  Erd\H{o}s \cite{E46} showed in 1946 that $|\Delta(E)|\gtrsim \sqrt{n}$, and conjectured that a square lattice is the optimal configuration to minimize the size of the distance set, i.e.
$$
|\Delta(E)|\gtrsim \frac{n}{\sqrt{\log{n}}}.
$$
Here, and throughout, we use asymptotic notation $f(n)\lesssim g(n)$ interchangeably with $f(n)=O(g(n))$.  
\\

\n After almost 70 years and several iterations of improvement on the optimal bound, Guth and Katz \cite{GK15} essentially resolved the conjecture, up to a factor of $\sqrt{\log{n}}$.

\begin{theorem}[Guth-Katz]
For $E\subseteq \mathbb{R}^2$, $|E|=n$,
$$
|\Delta(E)|\gtrsim \frac{n}{\log{n}}.
$$
\end{theorem}

\n Iosevich and Rudnev \cite{IR07} studied a related problem over finite fields in 2007.  In the finite setting, in order for these asymptotic questions to have meaning, we generally consider $\mathbb{F}_q$, a finite field with $q=p^r$ elements, as $q\to \infty$.  Of course, it is not immediately clear what is meant by distances over finite fields in the first place.  
\v

\begin{definition}
We take the distance between $x,y\in \mathbb{F}_q^d$ to be
$$
||x-y||:=x_1^2+\cdots +x_d^2,
$$
\end{definition}
\v

\begin{remark}
This distance is an element of $\mathbb{F}_q$ and not a real number.  This does not yield a norm or a metric space, yet it shares many properties with the usual Euclidean sphere algebraically and geometrically.
\end{remark}
\v

\n We notice that the distance set can never be larger than the field itself, so a reasonable formulation of the Erd\H{o}s distance problem over finite fields is to ask how large $E\subseteq \mathbb{F}_q^2$ must be (or $E\subseteq \mathbb{F}_q^d$) to ensure that $\Delta(E)=\mathbb{F}_q$ recovers the entire field?
\\

\begin{theorem}[Iosevich-Rudnev]
If $E\subseteq \mathbb{F}_q^d$, $|E|\geq cq^{\frac{d+1}{2}}$, $c$ sufficiently large, it follows that
$$
\Delta(E)=\mathbb{F}_q.
$$ 
\end{theorem}
\n Indeed, they also demonstrated a more precise quantitative formulation of this result, counting the number of pairs $(x,y)\in E$ with $||x-y||=t$ for each $t\neq 0$.  We will make use of parts of this argument later, see for example Theorem \ref{IRthm}.
\\

\begin{definition}
Let $\mathcal{G}_t(E)$ be the distance-$t$ graph on $E\subseteq \mathbb{F}_q^d$, i.e. there is an edge $x\sim y$ whenever $||x-y||=t$.
\end{definition}
\v
\n Thus, the above discussion can be formulated in terms of counting edges in $\mathcal{G}_t(E)$, or equivalently counting embeddings of the graph with two vertices and one edge into $\mathcal{G}_t(E)$.  It is natural to extend this study to counting embeddings $G\to\mathcal{G}_t(E)$ for other graphs $G$.  Indeed, this was done for paths of arbitrary length by Bennett, Chapman, Covert, Hart, Iosevich, and Pakianathan \cite{BCCHIP16}, and was done for cycles of arbitrary length and for trees by the second listed author in joint work with Iosevich and Jardine \cite{IJM21}.  The tools used to handle sparse graphs like paths and cycles are somewhat different from the tools used for dense graphs, e.g. the complete graph $K_n$.  Iosevich and Parshall were able to formulate a very general framework which can apply to arbitrary graphs $G$, with bounds depending on the number of vertices and the maximum vertex degree.  In general, taking an approach which applies to a wide family of graphs generally results in weaker bounds than taking an approach specifically tailored to a narrower family of graphs.
\\

\n Another generalization of the Erd\H{o}s distance problem over finite fields is to consider the generalized distance set $\Delta_G(E)$ for a given graph $G$.

\v

\begin{definition}
Given a graph $G$ with vertex set $V=\{v_1,...,v_n\}$ and edge set $\mathcal{E}$, and a set $E\subseteq \mathbb{F}_q^d$, let
$$
\Delta_G(E):=\{(t_{ij})_{ij\in \mathcal{E}}: \exists x^1,...,x^n\in E, \ ||x^i-x^j||=t_{ij} \ \text{whenever} \ v_i\sim v_j\}
$$
    
\end{definition}

\v

\n This has been studied in the context $G=K_n$ by Bennett, Hart, Iosevich, Pakianathan, and Rudnev \cite{BHIPR17}, and in the context when $G$ is a tree by \cite{IJM21}.  The approaches are quite different, as trees are maximally flexible and complete graphs are maximally rigid.  The second listed author, with Aksoy and Iosevich \cite{AIM24}, developed an interpolation scheme utilizing Fourier restriction to handle cases of graphs with both rigid and flexible components.  This technique was extended by the SMALL REU 2024 \cite{SMALL24} to a wider class of graphs.
\\

\subsection{VC-dimension and learning theory} One family of graphs which has been increasingly studied in the context of distance problems in the last few years arises from the notion of VC-dimension, developed by Vapnik and Chervonenkis \cite{VC71} in 1968 in the context of machine learning theory, particularly the paradigm of PAC learning (probably approximately correct).  We will introduce the central ideas here, for a more thorough introduction to the subject see for example \cite{DS14}.

\begin{definition}
    Let X be a set and $\cH$ collection of functions from X to $\{0,1\}.$ We say that $\cH$ shatters a finite set $C \subset X$ if the restriction of $\cH$ to $C$ yields every function from C to $\{0,1\}.$
\end{definition}

\begin{definition}
    Let X and H be defined as above. We say that a non-negative integer n is the VC-Dimension of $\cH$ if there exists a $C \subset X$ of size n such that $\cH$ shatters C, and no subset of size $n+1$ is shattered by $\cH$.
\end{definition}

\begin{remark}
Shattering is a monotone condition, i.e. if $C'\subseteq C$ and $\mathcal{H}$ shatters $C$, then $\mathcal{H}$ shatters $C'$ as well.  Moreover if $\mathcal{H}\subseteq \mathcal{H}'$ and $\mathcal{H}$ shatters $C$, then $\mathcal{H}'$ shatters $C$ as well.
    
\end{remark}

\begin{definition}
Let $G_n$ be the graph which represents shattering $n$ points, i.e. $G_n$ is the graph with vertex set
$$
\{x^1,...,x^n\}\cup\{y^I: I\subseteq[n]\},
$$
\n with an edge $x^i\sim y^I$ whenever $i\in I$.  See for example $G_3$ in Figure \ref{vc3_graph}.
    
\end{definition}

\begin{figure}[h!]
\centering

\begin{tikzpicture}[scale=2, thick]
  \node[fill=black, circle, minimum size=4pt, inner sep=0pt, label=below:$y^{123}$] (y123) at (0,0) {};
  \node[fill=black, circle, minimum size=4pt, inner sep=0pt, label=below left:$x^2$] (x2) at (0,1) {};
  \node[fill=black, circle, minimum size=4pt, inner sep=0pt, label=right:$x^1$] (x1) at (-.707,.707) {};
  \node[fill=black, circle, minimum size=4pt, inner sep=0pt, label=above right:$x^3$] (x3) at (.707,.707) {};
  \node[fill=black, circle, minimum size=4pt, inner sep=0pt, label=above left:$y^{12}$] (y12) at (-.707,1.707) {};
  \node[fill=black, circle, minimum size=4pt, inner sep=0pt, label=above right:$y^{23}$] (y23) at (.707,1.707) {};
  \node[fill=black, circle, minimum size=4pt, inner sep=0pt, label=above:$y^{13}$] (y13) at (0,1.414) {};
  \node[fill=black, circle, minimum size=4pt, inner sep=0pt, label=above left:$y^1$] (y1) at (-1.707,.707) {};
  \node[fill=black, circle, minimum size=4pt, inner sep=0pt, label=above right:$y^3$] (y3) at (1.707,.707) {};
  \node[fill=black, circle, minimum size=4pt, inner sep=0pt, label=below right:$y^2$] (y2) at (.707,.293) {};
  \node[fill=black, circle, minimum size=4pt, inner sep=0pt, label=above left:$y^\varnothing$] (y0) at (-1.707,1.707) {};
  \node at (0,2) {\LARGE $G_3$};

  \draw (y123) -- (x1);
  \draw (y123) -- (x2);
  \draw (y123) -- (x3);
  \draw (y12) -- (x1);
  \draw (y12) -- (x2);
    \draw (y13) -- (x1);
    \draw (y13) -- (x3);
    \draw (y23) -- (x2);
    \draw (y23) -- (x3);
    \draw (y1) -- (x1);
    \draw (y2) -- (x2);
    \draw (y3) -- (x3);
\end{tikzpicture}

\caption{The graph $G_3$ represents shattering 3 points.}\label{vc3_graph}
\end{figure}
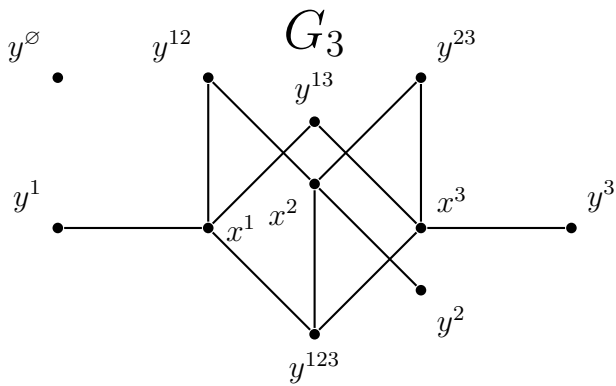
\v

\n Haussler and Welzl \cite{HW87} demonstrated examples of set systems whose VC-dimension is determined by their geometry in 1987. In 1995 Haussler \cite{H95} showed that assuming bounded VC-dimension of a subset of the Boolean cube yields improved bounds for sphere packing with respect to the hamming distance, compared to the result of Dudley \cite{D78} in 1978.  More recently, in 2019 and 2021 Fox, Pach, and Suk \cite{FPS19, FPS21} showed that the Schur-Erd\H{o}s conjecture, as well as a slightly weaker version of the Erd\H{o}s-Hajnal conjecture, both hold for graphs with bounded VC-dimension.  In 2025 Nguyen, Scott, and Seymour \cite{NSS25} strengthened this and showed that the Erd\H{o}s-Hajnal conjecture holds for graphs of bounded VC-dimension.
\v

\n In recent years there have also been a handful of results on VC-dimension of hypothesis classes defined over finite fields, in the context of distance sets, hyperplanes, and quadratic residues, and pseudo-random graphs \cite{SMALL22,SMALL23,FIMW23,IMS23,MSW,PSTV22,PSTT25}.  These results generally take the form of defining a hypothesis class $\mathcal{H}(E)$ depending on a set $E\subseteq \mathbb{F}_q^d$, and determining thresholds on the size of $E$ to guarantee that $\mathcal{H}(E)$ has the same VC-dimension as the ambient space $\mathbb{F}_q^d$.  Recently the second listed author, in joint work with Iosevich, Magyar, and A. McDonald \cite{IMMM25}, extended these techniques to fractals in Euclidean space as well.  In this context one seeks a threshold for the Hausdorff dimension of $E\subseteq \mathbb{R}^d$, rather than cardinality.  
\v

\n

\n Our new contribution here is as follows: The condition that $||x-y||=t$ can be reformulated as $x-y\in S$, where $S$ is the circle of radius $t$.  The central question we seek to answer here is how the analysis changes by considering different sets $S$.  In this investigation, we need to understand how the techniques of these other papers depend on specific properties of circles and spheres.  In particular, we find that most of these results boil down to having good bounds on the Fourier transform of a sphere, understanding intersection properties of spheres, circles, hyperplanes, etc., and having precise estimates for the number of points on a sphere.  Theorem \ref{shatter3} demonstrates that a many of these results can be derived purely from these properties.  We are then able to extend the constructions of \cite{FIMW23} not only to a larger class of algebraic curves and surfaces, but also to random sets as in Theorem \ref{VC_random}.
\\

\n On the other hand, many of the results in this area rely on symmetry of spheres in an essential way by making use of the action of the orthogonal group.  Because of this, we explicitly cannot extend the techniques of \cite{AIM24,BHIPR17,SMALL24,M20} to this context.
\v

\subsection{Discrete Fourier analysis preliminaries}

We will need some basic facts about Fourier analysis over finite fields, and here we record our normalization conventions.
\v

\begin{definition}
    Let $\chi$ be a nontrivial additive character on $\mathbb{F}_q$, so that every additive character over $\mathbb{F}_q^d$ can be represented in the form
    $$
    \chi_m(x)=\chi(m\cdot x).
    $$
    The choice of nontrivial $\chi\in \hat{\mathbb{F}}_q$ is arbitrary, but will remain fixed throughout.
\end{definition}
\v

\begin{proposition}[Orthogonality of characters]
Let $\delta_0$ be the Dirac delta.  Then,
$$
\sum_{m\in \mathbb{F}_q^d}\chi(m\cdot x)=q^d\delta_0(x)=\left\{\begin{array}{ll}
q^d & :x=0 \\
0 & :x\neq 0
\end{array}\right.
$$
    
\end{proposition}
\v

\begin{definition}
For a function $f:\mathbb{F}_q^d\to \mathbb{C}$, the Fourier transform of $f$ is defined by
\v

$$
\hat{f}(m)=q^{-d}\sum_{x\in \mathbb{F}_q^d}\chi(-m\cdot x)f(x).
$$
    
\end{definition}
\v

\begin{proposition}[Fourier inversion]
    We have the inversion formula
    \v

    $$
    f(x)=\sum_{m\in \mathbb{F}_q^d}\chi(m\cdot x)\hat{f}(m).
    $$
\end{proposition}
\v

\begin{proposition}[Plancharel identity]
$$
\sum_{x\in \mathbb{F}_q^d}|f(x)|^2=q^d\sum_{m\in \mathbb{F}_q^d}|\hat{f}(m)|^2.
$$
    
\end{proposition}
\v

\subsection{Results}
\v

\begin{definition}
For $E\subseteq \mathbb{F}_q^d$, consider the hypothesis class $\cH_S(E)$, with domain $X=E$, defined by
$$
\cH_S(E):=\{h_y: y\in E\}, \ \ \ h_y(x):=\left\{\begin{array}{ll}
1 & x-y\in S \\
0 & x-y\notin S
\end{array}\right.
$$
    
\end{definition}

\begin{definition}
We say that $S\subseteq \mathbb{F}_q^d$ is a $\gamma$-Salem set if for all $m\neq 0$,
$$
|\hat{S}(m)|\leq q^{-d}(\log{q})^{\gamma}|S|^{\frac{1}{2}}.
$$
In the case $\gamma=0$, we simply call it a Salem set.
\end{definition}
\v

\begin{remark}
We will need this notion of a $\gamma$-Salem set for the case when $S$ is a randomly chosen set.  For all the other results we could simply work with Salem sets and not keep track of this logarithmic factor.  However, it is convenient to take a unified approach which handles both at the same time, and one can always set $\gamma=0$ and ignore the logarithmic factor in settings when it is not needed.
    
\end{remark}

\begin{definition}
Given $E,S\subseteq \mathbb{F}_q^d$, let $\mathcal{G}_S(E)$ be the directed graph with vertex set $E$, and with an edge $x\sim y$ whenever $x-y\in S$.  In the case when $S$ is symmetric, i.e. $S=-S$, we may consider $\mathcal{G}_S(E)$ to be an undirected graph.

\end{definition}

\begin{theorem}\label{shatter3}
    Suppose $S \subset \bF_q^d$ is a $\gamma$-Salem set, symmetric in the sense $S=-S$, $|S|=Kq^{d-1}(1+O(q^{-\alpha}))$ for a constant $K$, $\alpha>\frac{d-1}{2}$, and $|S\cap (S-x)|\leq 2q^{\beta}$, $\beta<\frac{d-1}{2}$, for all nonzero $x\in \mathbb{F}_q^d$. If $E \subset \bF_q^d$ with $|E| \ge cq^\frac{7d+1}{8}$, for sufficiently large constant c. Then the VC-dimension of $\mathcal{H}_S^2(E)$ is at least $3$.
\end{theorem}

\begin{corollary}\label{VC3}
If $S$ and $E$ satisfy the conditions of Theorem \ref{shatter3}, and $\beta=0$, then we can conclude that
$$
\text{VCdim}(\mathcal{H}_S(E))=3.
$$
    
\end{corollary}

\begin{proof}
    we need to check that it is impossible to shatter 4 points.  This is straightforward, since if we have some shattering configuration
$$
\{x^i: 1\leq i \leq 4\}\cup \{y^I: I\subseteq [4]\}\subseteq E,
$$
such that $x^i-y^I\in S$ if and only if $i\in I$, we see that
$$
\{x^1,x^2,x^3\}\subseteq(y^{1234}+S)\cap(y^{123}+S),
$$
contradicting that this intersection can have at most 2 points.
\end{proof}
\v

\n In Section \ref{examples}, we will demonstrate some applications of Theorem \ref{shatter3}.  In particular, we generalize \cite{FIMW23} to hypothesis classes over $\mathbb{F}_q^2$ defined by arbitrary quadratic curves, as opposed to circles.  The same framework allows us to consider random sets.

\begin{theorem}\label{VC_random}
Let $S$ be chosen uniformly at random from the subsets of $\mathbb{F}_q^2$ of size $q$.  Let $T=S\cup (-S)$.  Then, with probability $1-o(1)$,
$$
\text{VCdim}(\mathcal{H}_S(E))\geq 2 \ \ \ \text{and} \ \ \ \text{VCdim}(\mathcal{H}_T(E))\geq 3.
$$
    
\end{theorem}
\v

\begin{theorem}\label{VC_quadratic}
Let $E\subseteq \mathbb{F}_q^2$, $|E|\geq cq^{\frac{15}{8}}$ for $c$ sufficiently large.  Let
$$
f(x,y)=Ax^2+Bxy+Cy^2+Dx+Ey+F,
$$
and let $Z_f$ be the zero set of $f$. Suppose that
$$
\det\left(\begin{array}{ccc}
A & B/2 & D/2 \\
B/2 & C & E/2 \\
D/2 & E/2 & F
\end{array}
\right)\neq 0.
$$
If 
$$
\det\left(\begin{array}{cc}
A & B/2 \\
B/2 & C
\end{array}\right)\neq 0,
$$
Then there exists a translation $S$ of $Z_f$ such that
$$
2\leq \text{VCdim}(\mathcal{H}_{Z_f}(E))\leq \text{VCdim}(\mathcal{H}_S(E))=3.
$$
Otherwise, if the quadratic part is degenerate, the VC-dimension is equal to 2 for any translation.  In this case, 
$$
\text{VCdim}(\mathcal{H}_T(E))\geq 3,
$$
where $T=S\cup (-S)$.
\end{theorem}
\v

\begin{remark}
In the setting of Theorem \ref{VC_quadratic}, when the quadratic part of $f$ is non-degenerate, we do not know in general whether $\text{VCdim}(\mathcal{H}_{Z_f}(E))$ is equal to 2 or 3.  We suspect that it might always be equal to 3, but the techniques we develop in Section \ref{main_section} depend on symmetry in an essential way, and so this is left as an open question.
    
\end{remark}
\v

\n Finally, we have one more application of Theorem \ref{shatter3}.

\begin{theorem}\label{odd_function}
Let $E\subseteq \mathbb{F}_q^2$, $|E|\geq cq^{\frac{15}{8}}$ for $c$ sufficiently large.  Consider a degree $n$ polynomial $f\in \mathbb{F}_q[x]$, $p\nmid n$, and let $\Gamma_f:=\{(x,f(x)): x\in \mathbb{F}_q\}$ be the graph of the function.  Then $\text{VCdim}(\mathcal{H}_{\Gamma_f}(E))\geq 2$, and if $f$ is an odd function then $\text{VCdim}(\mathcal{H}_{\Gamma_f}(E))\geq 3$.
\end{theorem}
\v

\section{Examples of Salem sets over finite fields}\label{examples}
\v

\subsection{Explicit examples}
Two explicit examples of Salem sets in $\mathbb{F}_q^d$ are given by \cite{IR07}.
\v

\begin{theorem}[Iosevich-Rudnev]\label{IRthm}
    The sphere of radius $t$, 
$$
\mathcal{S}_t=\{x\in \mathbb{F}_q^d: x_1^2+\cdots +x_d^2=t\},
$$
and the paraboloid
$$
\mathcal{P}=\{(x_1,x_2,\dots,x_{d-1},x_1^2+\cdots + x_{d-1}^2): x_1,\dots,x_{d-1}\in \mathbb{F}_q\}
$$
are both Salem sets.  
\end{theorem}
\v

\n We will provide a few more examples and discuss how they relate to Theorem \ref{shatter3}.
\v

\begin{lemma}\label{linear}
Suppose that $S\subseteq \mathbb{F}_q^d$ is obtained from $S_0\subseteq \mathbb{F}_q^d$ via an invertible linear transformation, i.e.
$$
S(x)=S_0(Tx).
$$
Then $S$ is a $\gamma$-Salem set if and only if $S_0$ is.
\v
    
\end{lemma}
\begin{proof}
We will assume that $S_0$ is a Salem set and show that $S$ must be as well, noting that the other direction follows by symmetry.  We can compute directly, where $T^{-t}$ is the inverse transpose of $T$ and $m\neq 0$,
$$
\hat{S}(m)=q^{-d}\sum_{x\in \mathbb{F}_q^d}\chi(-m\cdot x)S_0(Tx)
=q^{-d}\sum_x\chi(-m\cdot (T^{-1}x))S_0(x)
$$
$$
=q^{-d}\sum_x \chi(-x\cdot (T^{-t}m))S_0(x)
=\hat{S_0}(T^{-t}m)
$$
    
\end{proof}
\v

\begin{lemma}
Suppose that $S\subseteq \mathbb{F}_q^d$ is obtained from $S_0\subseteq \mathbb{F}_q^d$ via a translation, i.e. $S(x)=S_0(x+v)$ for some vector $v\in \mathbb{F}_q^d$.  Then $S$ is a $\gamma$-Salem set if and only if $S_0$ is.
    
\end{lemma}
\v

\begin{proof}
This is another straightforward calculation:
$$
\hat{S}(m)=q^{-d}\sum_x \chi(-m\cdot x)S_0(x+v)
=q^{-d}\sum_x \chi(-m\cdot (x-v))S_0(x)
$$
$$
=\chi(m\cdot v)\hat{S_0}(m),
$$
and so $|\hat{S}(m)|=|\hat{S_0}(m)|$.
    
\end{proof}
\v

\begin{lemma}\label{quadratic_form}
Any degree two polynomial $f\in \mathbb{F}_q[x,y]$, $f\notin\mathbb{F}_q[x]\cup\mathbb{F}_q[y]$, $q$ odd, can be reduced to one of the following forms via an invertible linear change of variables and a translation:
\v

\begin{enumerate}
\item $f(x,y)=y-x^2$

\item $f(x,y)=ax^2+by^2+c$
    
\end{enumerate}
\v

\n Moreover, it takes the first form if the quadratic part is a non-degenerate quadratic form, and takes the second form if it is degenerate.  If $f$ takes the second form and is a smooth curve, then $c\neq 0$.
    
\end{lemma}
\v

\begin{proof}
Let
$$
f(x,y)=Ax^2+Bxy+Cy^2+Dx+Ey+F.
$$
Consider the quadratic form $Q(x,y)=Ax^2+Bxy+Cy^2$, which corresponds to the symmetric matrix
$$
M=\left(\begin{array}{cc}
A & \frac{B}{2} \\
\frac{B}{2} & C
\end{array}
\right)
$$
in the sense that for any column vector $v=\left(\begin{array}{c} x \\ y \end{array}\right)$,
$$
Q(x,y)=Q(v)=v^tMv.
$$
  This is where we have used the assumption $q$ is odd, so that $\frac{1}{2}\in \mathbb{F}_q$.  $M$ is a real symmetric matrix and hence is diagonalizable.  Therefore there is an orthogonal transformation $U$ such that $U^t M U$ is a diagonal matrix.  This means that we may assume without loss of generality that $B=0$, since otherwise it can be eliminated via the change of coordinates
$$
v'=\left(\begin{array}{c}
x' \\
y'
\end{array}\right)
=U^t\left(\begin{array}{c}
x \\
y
\end{array}\right)=U^t v.
$$
Let $Q'$ be the quadratic form corresponding to the matrix $U^t M U$.  Under this transformation,
$$
Q'(v')=(v')^t(U^tMU)v'
=(Uv')^tM(Uv')
=v^t M v
=Q(v).
$$
If $M$ is invertible, let
$$
w=-\frac{1}{2}M^{-1}\left(\begin{array}{cc}
D \\ E
\end{array}\right),
$$
so that
$$
f(x,y)=f(v)=v^tMv+Dx+Ey+F
$$
$$
=(v-w)^tM(v-w)+w^tMv+v^tMw-w^tMw+\left(\begin{array}{c} D \\ E \end{array}\right)^t (v-w)+\left(\begin{array}{c} D \\ E \end{array}\right)^t w+F
$$
$$
=Q(v-w)+\left(2Mw+\left(\begin{array}{c} D \\ E \end{array}\right)\right)^t\cdot (v-w)+w^tMw+ \left(\begin{array}{c} D \\ E \end{array}\right)^tw+F
$$
$$
=Q(v-w)+F',
$$
where
$$
F'=w^tMw+\left(\begin{array}{c} D \\ E \end{array}\right)^tw+F.
$$
\v

\n Since $F'$ is a constant vector, not depending on the variables $x$ and $y$, we see that translation by $w$ eliminates linear terms assuming $M$ is invertible.  This combined with the previous discussion of diagonalizing $M$ proves the Lemma in the case when $M$ is invertible.
\\

\n If $M$ is not invertible, Then there is some nonzero vector $u$ such that $Mu=0$.  Without loss of generality assume that $u=\left(\begin{array}{c} 0 \\ 1 \end{array}\right)$, since otherwise we may apply an orthogonal transformation mapping $\left(\begin{array}{c} 0 \\ 1 \end{array}\right)$ to $u$.  This change of basis does not affect whether the linear transformation corresponding to $M$ is invertible.  If $M\left(\begin{array}{c} 0 \\ 1 \end{array}\right)=0$, then $B=C=0$, and so 
$$
f(x,y)=Ax^2+Dx+Ey+F.
$$
Since we assumed $f\notin\mathbb{F}_q[x]$, we know that $E\neq 0$, so we may express this in the standard quadratic form
$$
y=ax^2+bx+c.
$$
Then, via translation, we get $y=x^2$.
\end{proof}
\v

\n We still need to show that $Z_f$ satisfies the hypotheses of Theorem \ref{shatter3}.  We will need the well-known Weil bound for Kloosterman sums.  We will also later use another form of the Weil bounds  \cite{W48}, so we will introduce both here.  See, for example, Iwaniec and Kowalski \cite{IK} for an overview of the various forms of Weil bounds for character sums over finite fields.
\v

\begin{theorem}[Weil, 1948]\label{weil}
For $a,b\in \mathbb{F}_q^*$,
$$
\left|\sum_{j\in \mathbb{F}_q}\chi(aj+bj^{-1})\right|\leq 2\sqrt{q}.
$$
Moreover, if $f\in \mathbb{F}_q[x]$ is a polynomial of degree $n$, $f$ not of the form $c+h^p-h$ for any polynomial $h\in \mathbb{F}_q[x]$, then
$$
\left|\sum_{j\in \mathbb{F}_q}\chi(f(j))\right|\leq (n-1)\sqrt{q}.
$$
\end{theorem}
\v

\begin{lemma}\label{salem_quadratic}
Let $f$ be a degree 2 polynomial $f\in \mathbb{F}_q[x,y]$, $f\notin \mathbb{F}_q[x]\cup \mathbb{F}_q[y]$, $q$ odd.  As long as the reduction from Lemma \ref{quadratic_form} does not yield the equation $ax^2+by^2=0$, it follows that $Z_f$ is a Salem set with $|Z_f|$ equal to $q-1$, $q$, or $q+1$. 
    
\end{lemma}
\v

\begin{proof}
A parabola $y=x^2$ clearly has exactly $q$ solutions.  We know from \cite{IR07} that the parabola is a Salem set.  It remains to consider the case
$$
Z_f=S_{a,b}:=\{(x,y): ax^2+by^2=1\}.
$$
Viewing this as the projective curve $aX^2+bY^2=Z^2$, there are exactly $q+1$ points.  This is because for a fixed point $P$ on the curve, every line through $P$ intersects the projective curve in exactly one other point.  This yields a bijection between the curve and a projective line, which has $q+1$ points.  To find the number of points on the curve in affine space, observe that in projective space the line at infinity intersects the curve in at most 2 points, and so the number of points on the curve in affine space is one of $q-1$, $q$, or $q+1$.
\\

\n We still need a pointwise bound on the Fourier transform $|\hat{S}_{a,b}(m)|$ for $m\neq 0$.

$$
\hat{S}_{a,b}(m)=q^{-2}\sum_{x\in \mathbb{F}_q^2}\chi(-m\cdot x)S_{a,b}(x)
=q^{-3}\sum_{x\in \mathbb{F}_q^2}\sum_{j\in \mathbb{F}_q}\chi(-m\cdot x)\chi(j(ax_1^2+bx_2^2-1)
$$
$$
=q^{-3}\sum_{j\in \mathbb{F}_q^*}\chi(-j)\sum_{x\in \mathbb{F}_q^2}\chi(jax_1^2+jbx_2
^2-m\cdot x)
$$
This last step is where we have used the fact $m\neq 0$, for this ensures that the $j=0$ term vanishes.  For this inner sum, we have
$$
\sum_x \chi(jax_1^2+jbx_2^2-m\cdot x)
=\left(\sum_{x_1}\chi(jax_1^2-m_1x_1)\right)\left(\sum_{x_2}\chi(jbx_2^2-m_2x_2)\right),
$$
where
$$
\sum_{x_1\in \mathbb{F}_q}\chi(jax_1^2-m_1x_1)
=\chi\left(-\frac{m_1^2}{4ja}\right)\sum_{x_1}\chi\left(ja\left(x_1-\frac{m_1}{2ja}\right)^2\right),
$$
and
$$
\sum_{x_1}\chi\left(ja\left(x_1-\frac{m_1}{2ja}\right)^2\right)=\sum_{x_1}\chi(jax_1^2)
$$
\v

\n is a Gauss sum, for which it is well known for $k\neq 0$,
\v

$$
g(k):=\sum_{x\in \mathbb{F}_q}\chi(kx^2)=\e_q\eta(k)\sqrt{q},
$$
\v

\n where $\e_q\in \{\pm 1, \pm i\}$ depends on the field but not on $k$, and
\v

$$
\eta(k)=\left\{\begin{array}{ll}
1 & : k\neq 0 \ \text{is a square in} \ \mathbb{F}_q \\
-1 & :k\neq 0 \ \text{is not a square in} \ \mathbb{F}_q \\
0 & : k=0
\end{array}\right.
$$
\v

\n Therefore, putting this all together we obtain
\v

$$
\hat{S}_{a,b}(m)
=q^{-3}\sum_{j\in \mathbb{F}_q^*}\chi(-j)\chi\left(\frac{1}{j}\left(-\frac{m_1^2}{4a}-\frac{m_2^2}{4b}\right)\right)g(ja)g(jb)
$$
$$
=q^{-2}\e_q^2\sum_{j\in \mathbb{F}_q^*}\chi(-j)\chi\left(\frac{1}{j}\left(-\frac{m_1^2}{4a}-\frac{m_2^2}{4b}\right)\right)\eta(j^2ab)
$$
$$
=q^{-2}\e_q^2\eta(ab)\sum_{j\in \mathbb{F}_q^*}\chi(-j)\chi\left(\frac{1}{j}\left(-\frac{m_1^2}{4a}-\frac{m_2^2}{4b}\right)\right).
$$
\v

\n Finally, by Theorem \ref{weil} we see that 
\v

$$
|\hat{S}_{a,b}(m)|\leq 2q^{-\frac{3}{2}}
$$
\v

\n as long as $\frac{m_1^2}{4a}+\frac{m_2^2}{4b}\neq 0$.  If $\frac{m_1^2}{4a}+\frac{m_2^2}{4b}=0$, $m\neq 0$, then we can compute exactly
\v

$$
|\hat{S}_{a,b}(m)|=q^{-2}<q^{-\frac{3}{2}}.
$$
\end{proof}
\v

\n We are almost ready to prove Theorem \ref{VC_quadratic}, we just need one more Lemma.  We will defer the proof of Lemma \ref{shatter2} until Section \ref{main_section}, after we develop some more Fourier analytic machinery.
\v

\begin{lemma}\label{shatter2}
Suppose $S\subseteq \mathbb{F}_q^d$ is a $\gamma$-Salem set, $|S|=Kq^{d-1}(1+O(q^{-\alpha}))$ for a constant $K$ and for $\alpha>\frac{d-1}{2}$, and $|S\cap (S-x)|\leq q^{\beta}$, $\beta<\frac{d-1}{2}$, for all nonzero $x\in \mathbb{F}_q^d$.  If $E\subseteq \mathbb{F}_q^d$, $|E|\geq cq^{\frac{d+1}{2}}$ for $c$ sufficiently large, then 
$$
\text{VCdim}(\mathcal{H}_S(E))\geq 2.
$$
    
\end{lemma}
\v

\n This corresponds to the configuration in Figure \ref{path4}.

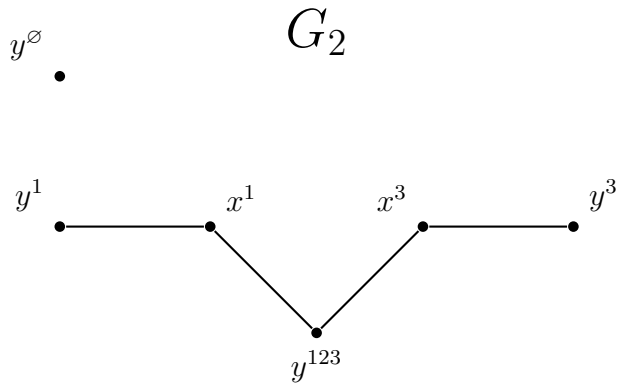
\begin{figure}[h!]
\centering
\begin{tikzpicture}[scale=2, thick]
  \node[fill=black, circle, minimum size=4pt, inner sep=0pt, label=below:$y^{123}$] (y123) at (0,0) {};
  
  \node[fill=black, circle, minimum size=4pt, inner sep=0pt, label=above right:$x^1$] (x1) at (-.707,.707) {};
  \node[fill=black, circle, minimum size=4pt, inner sep=0pt, label=above left:$x^3$] (x3) at (.707,.707) {};

  \node[fill=black, circle, minimum size=4pt, inner sep=0pt, label=above left:$y^1$] (y1) at (-1.707,.707) {};
  \node[fill=black, circle, minimum size=4pt, inner sep=0pt, label=above right:$y^3$] (y3) at (1.707,.707) {};
  
  \node[fill=black, circle, minimum size=4pt, inner sep=0pt, label=above left:$y^\varnothing$] (y0) at (-1.707,1.707) {};
  \node at (0,2) {\LARGE $G_2$};

  \draw (y123) -- (x1);
 
  \draw (y123) -- (x3);
  
    \draw (y1) -- (x1);
   
    \draw (y3) -- (x3);
\end{tikzpicture}

\caption{The graph $G_2$ represents shattering 2 points.}\label{path4}
\end{figure}
\v

\n We now have all the tools we need to prove Theorem \ref{VC_quadratic}.
\v

\begin{proof}[proof of Theorem \ref{VC_quadratic}]
This is a direct application of Theorem \ref{shatter3}, Lemma \ref{quadratic_form}, and Lemma \ref{salem_quadratic}.  Since we require
\v

$$
\det\left(\begin{array}{ccc}
A & B/2 & D/2 \\
B/2 & C & E/2 \\
D/2 & E/2 & F
\end{array}
\right)\neq 0,
$$
\v

\n it follows that the curve is smooth.  To see this, homogenize $f(x,y)$ to obtain
\v

$$
F(X,Y,Z)=Ax^2+Bx+Cy^2+Dxz+Eyz+Fz^2,
$$
\v

\n and note that the system of equations
\v

$$
0=\nabla F(X,Y,Z)=\langle 2Ax+By+Dz,  \ Bx+2Cy+Ez, \ Dx+Ey+2Fz\rangle
$$
\v

\n has a nontrivial solution if and only if the above matrix is invertible.  If 
\v

$$
\det\left(\begin{array}{cc}
A & B/2 \\
B/2 & C
\end{array}\right)\neq 0,
$$
\v

\n then by Lemma \ref{quadratic_form} we may assume without loss of generality that $f(x,y)=ax^2+by^2+c$.  Since the curve is smooth, we know $c\neq 0$.  We claim that for any vector $(h,k)\in \mathbb{F}_q^2$,
\v
$$
f(x,y)=f(x+h,y+k)=0
$$
\v
\n has at most 2 solutions.  To see this, note that
\v
$$
f(x+h,y+k)-f(x,y)=0
$$
\v
\n is a line, the empty set, or the whole space in the case when $(h,k)=(0,0)$.  Therefore, 
\v
$$
|Z_f\cap (Z_f-(h,k))|\leq 2
$$
\v
\n whenever $(h,k)\neq (0,0)$, since a line intersects a quadratic curve in at most 2 points.  Clearly it is possible to find a translation $S$ of $Z_f$ such that $S=-S$.  This is achieved by completing the square to eliminate linear terms, as in Lemma \ref{quadratic_form}.  Therefore, $S$ satisfies all conditions of Corollary \ref{VC3}, and so we can conclude that
\v
$$
\text{VCdim}(\mathcal{H}_{S}(E))=3.
$$

\n To see that

$$
\text{VCdim}\mathcal{H}_{Z_f}(E)\leq 3,
$$
\v

\n note that the argument in Corollary \ref{VC3} that it is impossible to shatter 4 points still holds here, even though $Z_f$ does not necessarily satisfy all hypotheses of Theorem \ref{shatter3}.
\v

\n We still need to consider the case when
\v

$$
\det\left(\begin{array}{cc}
A & B/2 \\
B/2 & C
\end{array}\right)= 0,
$$
\v

\n and so $Z_f$ is a parabola.  In this case, there is no hope to achieve $S=-S$ for any translation of $Z_f$.  Indeed, we can see immediately that it is impossible to shatter 3 points, because $|S\cap (S-x)|\leq 1$ for all $x\neq 0$. Therefore we may make an argument similar to Corollary \ref{VC3}.  If 
\v

$$
\{x^1,x^2,x^3\}\cup \{y^I: I\subseteq [3]\}
$$
\v

\n is a shattering configuration, we get a contradiction based on the fact that 
\v

$$
\{x^1,x^2\}\subseteq (y^{123}+S)\cap (y^{12}\cap S).
$$
\v

\n Finally, we know by Lemma \ref{shatter2} that 
\v

$$
\text{VCdim}(\mathcal{H}_S(E))=2.
$$
\v

\n In this case, we obtain the result for the symmetrized parabola $P=S\cup (-S)$ via Corollary \ref{finish} later in the section.

\end{proof}
\v

\begin{lemma}\label{odd_lemma}
Consider a degree $n\geq 2$ polynomial $f\in \mathbb{F}_q[x]$, and let $\Gamma_f:=\{(t,f(t)): t\in \mathbb{F}_q\}$ be the graph of the function.  $\Gamma_f$ is a Salem set with exactly $q$ elements, and $|\Gamma_f\cap (\Gamma_f-x)|\leq n-1$ for all $x\in \mathbb{F}_q^2$.
    
\end{lemma}
\v

\begin{proof}
We immediately get $|\Gamma_f|=q$ and $|\Gamma_f\cap (\Gamma_f-x)|\leq n-1$ by the fact that $f$ is a polynomial of degree $n$.  Let
\v
$$
f(t)=a_0+a_1t+\cdots +a_nt^n.
$$
\v

\n To see that it is a Salem set, we compute for $m\neq 0$,
\v
$$
\hat{\Gamma}_f(m)=q^{-2}\sum_{x\in \mathbb{F}_q^2}\chi(-m\cdot x)\Gamma_f(x)
=q^{-3}\sum_{x\in \mathbb{F}_q^2}\sum_{r\in \mathbb{F}_q}\chi(-m\cdot x)\chi(r(f(x_1)-x_2))
$$
$$
=q^{-3}\sum_r\left(\sum_{x_2}\chi(-(r+m_2)x_2)\right)\left(\sum_{x_1} \chi(rf(x_1)-m_1x_1)\right)
$$
\v

\n Here, $g(t)=rf(t)-m_1t$ is a degree $n$ polynomial which is not of the form $c+h^p-h$ for any polynomial $h$, since the degree is not divisible by $p$. Therefore, by Theorem \ref{weil} we have
\v
$$
\hat{\Gamma}_f(m)\lesssim q^{-\frac{5}{2}}\sum_r\sum_{x_2}\chi(-(r+m_2)x_2)=q^{-\frac{3}{2}}
$$
\end{proof}
\v

\begin{corollary}
    The conclusion of Theorem \ref{odd_function} follows immediately from Lemma \ref{shatter2}, Lemma \ref{odd_lemma} and Theorem \ref{shatter3}, since $\Gamma_f=-\Gamma_f$ if $f$ is an odd function.
\end{corollary}

\subsection{Random sets}

For the remainder of Section \ref{examples}, we will consider the situation where $S\subseteq \mathbb{F}_q^2$ is randomly chosen, uniformly from the set of subsets of $\mathbb{F}_q^2$ of size $q$.  Throughout the section, for sets $S$ determined by such a random process, we will refer to them simply as random subsets of $\mathbb{F}_q^2$.  We will make use of the following result of Hayes \cite{H01}, which answers a question of Babai \cite{B02}.
\v

\begin{theorem}[Hayes]\label{random_salem}
Let $\e>0$.  Let $G$ be a finite abelian group of order $n$.  Let $k\leq n$, and let $m=\min(k,n-k)$.  For all but an $O(n^{-\e})$ fraction of subsets $S\subseteq G$ such that $|S|=k$,
\v
$$
\Phi(S):=\max_{\chi}\left|\sum_{x\in G}\chi(x)f(x)\right|
<2\sqrt{2(1+\e)m\log{n}},
$$
where this max is take over nontrivial characters of $G$.
\end{theorem}
\v

\begin{corollary}
Translating this to our notation in the case $G=\langle \mathbb{F}_q^2,+\rangle$, $S\subseteq \mathbb{F}_q^d$ chosen randomly with $|S|=q$, we see that with probablity $1-o(1)$, for all $m\neq 0$,
\v

$$
|\hat{S}(m)|\lesssim q^{-\frac{3}{2}}\sqrt{\log{q}}.
$$
\v

\n In other words, $S$ is a $\frac{1}{2}$-Salem set.
\end{corollary}
\v

\begin{lemma}\label{random_properties}
If $S\subseteq \mathbb{F}_q^2$ is a random subset, then $S$ satisfies the hypotheses of Theorem \ref{shatter3}, with the exception of the symmetry property $S=-S$, with probability $1-o(1)$ as $q\to \infty$.  In other words, with high probability, we have
\v

\begin{enumerate}
\item $|S|=q(1+O(q^{-\alpha}))$, $\alpha>\frac{1}{2}$.
\v

\item $S$ is a $\gamma$-Salem set with $\gamma=\frac{1}{2}$.
\v

\item $|S\cap (S-x)|\lesssim  q^{\beta}$, $\beta<\frac{d-1}{2}$ for all nonzero $x\in \mathbb{F}_q^d$.
    
\end{enumerate}
    
\end{lemma}
\v

\begin{proof}
(1) follows immediately from the definition of $S$, since $|S|=q$ exactly. (2) follows immediately from Theorem \ref{random_salem}. It remains to prove (3).  Fix $x\in \mathbb{F}_q^2$, and let 
\v

$$
\Omega^x
=|S\cap (S-x)|
=\sum_{y\in \mathbb{F}_q^2}S(y)S(x+y)
=\sum_{y\in \mathbb{F}_q^2}\Omega_y^x,
$$
\v

\n where $\Omega_y^x=1$ in the event that $y\in S\cap (S-x)$.  We see that
\v

$$
\mathbb{E}[\Omega_y^x]=P(\Omega_y^x=1)
=\frac{1}{q(q+1)},
$$
and so
$$
\mathbb{E}[\Omega^x]=\sum_{y}\mathbb{E}[\Omega_y^x]=\frac{q}{q+1}\leq 1.
$$
\v

\n Therefore by Markov's inequality,

$$
P(\Omega^x>q^{\beta})\leq q^{-\beta}=o(1).
$$
\end{proof}
\v

\begin{lemma}\label{symmetrize}
Suppose $T=S\cup (-S)$, and $S$ is a set, not necessarily random, which satisfies the three properties of Lemma \ref{random_properties}. Then as long as $|S\cap (-S)|=O(q^{1-\alpha}),$ $T$ satisfies properties (1) and (2), with the exception that we gain a factor of 2 in the first property, i.e.
\v

$$
|T|=2q(1+O(q^{-\alpha})).
$$
    
\end{lemma}
\v

\begin{proof}
Property (1) is trivial.  Property (2) follows by direct calculation.  Let 
\v

$$
R(x)=S(x)S(-x),
$$
\v

\n so that $T(x)=S(x)+S(-x)-R(x)$. Then,
\v

$$
|\hat{T}(m)|\lesssim q^{-\frac{3}{2}}\sqrt{\log{\gamma}}+|\hat{R}(m)|,
$$
\n where
\v

$$
|\hat{R}(m)|\leq \frac{|R|}{q^2}=O(q^{-1-\alpha}).
$$
    
\end{proof}

\n We are now ready to prove Theorem \ref{VC_random}.
\v

\begin{proof}[Proof of Theorem \ref{VC_random}]

This follows immediately from Theorem \ref{shatter3}, Lemma \ref{random_properties}, and Lemma \ref{symmetrize}, as long as we can prove that the intersection $T\cap (T-x)$ is small.  But this follows immediately by writing
\v

$$
T\cap (T-x)=
(S\cap (S-x))\cup (S\cap (-S-x))\cup (-S\cap (S-x))\cup (-S\cap (-S-x)).
$$
\v

\n The argument that showed $S\cap(S-x)$ is small with high probability applies to these other sets just the same.

\end{proof}
\v

\begin{corollary}\label{finish}
The symmetrized parabola $P=\{(x,x^2)\}\cup \{(x,-x^2)\}$ satisfies the hypotheses of Theorem \ref{shatter3}, and therefore the conclusion of Theorem \ref{VC_quadratic} follows.
    
\end{corollary}

\begin{proof}

    Since the symmetrized parabola $P=P^+\cup P^-$ satisfies the finite intersection property
$$
|P\cap (P-x)|\leq 6,
$$
and $P^+$ is a Salem set with $q$ elements, we see that $P$ satisfies the hypotheses of Theorem \ref{shatter3}.
\end{proof}
\v

\section{Proof of Theorem \ref{shatter3}}\label{main_section}
\v

\begin{lemma}\label{salem_edge}
Suppose that $S\subseteq \mathbb{F}_q^d$,
\v

$$
|\hat{S}(m)|\lesssim q^{-\frac{d+1}{2}}(\log{q})^{\gamma},
$$
\v

\n and $|S|=Kq^{d-1}(1+O(q^{-\alpha}))$ for some constant $K$, $\alpha>\frac{d-1}{2}$.  Then it follows that
\v

$$
\nu_S(E):=|\{(x,y)\in E^2: x-y\in S\}|=K\frac{|E|^2}{q}+\cE,
$$
where
$$
|\cE|\lesssim q^{\frac{d-1}{2}}(\log{q})^{\gamma}|E|
$$
\\
In particular, if $|E|\geq cq^{\frac{d+1}{2}}(\log{q})^{\gamma}$ for a sufficiently large constant $c$, it follows that 
\v

$$
\nu_S(E)\geq K\frac{|E|^2}{2q}
$$
\end{lemma}
\v

\begin{proof}
We generalize the method of Iosevich and Rudnev \cite{IR07}, who obtained this result in the special case when $S$ is a sphere of radius $t\neq 0$.  
\v

$$
\sum_{x,y\in \mathbb{F}_q^d}E(x)E(y)S(x-y)
=\sum_{x,y,m\in \mathbb{F}_q^d}E(x)E(y)\chi(m\cdot (x-y))\hat{S}(m)
=q^{2d}\sum_m\overline{\hat{E}(m)}\hat{E}(m)\hat{S}(m)
$$
\v

$$
=\frac{|E|^2 |S|}{q^d}+q^{2d}\sum_{m\ne 0}|\hat{E}(m)|^2\hat{S}(m)
=K\frac{|E|^2}{q}(1+O(q^{-\alpha}))+q^{2d}\sum_{m\neq 0}|\hat{E}(m)|^2\hat{S}(m)
$$
\v

\n The $O(q^{-\alpha})$ term is irrelevant, since for any $E\subseteq \mathbb{F}_q^2$, we have $\frac{|E|^{2}}{q^{1+\alpha}}\leq q^{\frac{d-1}{2}}|E|$.
To bound the terms with $m\neq 0$, we use the pointwise bound $|\hat{S}(m)|\lesssim q^{-\frac{d+1}{2}}$, followed by Cauchy-Schwarz and Plancharel.
\v

$$
q^{2d}\sum_{m\neq 0}|\hat{E}(m)|^2\hat{S}(m)
\lesssim q^{2d}q^{-\frac{d+1}{2}}(\log{q})^{\gamma}\sum_{m\in \mathbb{F}_q^2}|\hat{E}(m)|^2
$$
$$
=q^dq^{-\frac{d+1}{2}}(\log{q})^{\gamma}\left(\sum_x|E(x)|^2\right)
=q^{\frac{d-1}{2}}(\log{q})^{\gamma}|E|
$$
\end{proof}
\v

\n We can use this to demonstrate an abundance of tuples $(x^1,x^2,x^3,x^4)$ with vertices contained in $E\subseteq \mathbb{F}_q^d$, satisfying $x^1-x^2=x^3-x^4\in S$ and $x^1-x^3=x^2-x^4\in S$.  This generalizes the method of \cite{FIMW23}, which handles the special case where $S$ is a circle and so we would be constructing rhombi.
\v

\begin{lemma}\label{rhombus}
Suppose that $E\subseteq \mathbb{F}_q^d$, $|E|\geq cq^{\frac{3d+1}{4}}(\log{q})^{\gamma/2}$
for a sufficiently large constant $c$. Let $S\subseteq \mathbb{F}_q^d$ be a set satisfying the conditions of Lemma \ref{salem_edge}. Fix a non-zero vector $v\in \mathbb{F}_q^d$.  Then there exist distinct $x^1,x^2,x^3,x^4\in E$ such that
\v

$$
x^1-x^2=x^3-x^4\in S, \ \ x^1-x^3=x^2-x^4\in S.
$$
\v

\n and $x^i-x^j$ is not equal to $\pm v$ for any $i,j$.
    
\end{lemma}
\v

\begin{proof}

Observe that less than half of the pairs $(x,y)\in E^2$ satisfying $x-y\in S$ happen to also satisfy $x-y=\pm v$.  To see this,
\v

$$
|\{(x,y)\in E^2: x-y=\pm v\}|\leq 2|E|<K\frac{|E|^2}{4q}<\frac{1}{2}|\{(x,y)\in E^2: x-y\in S\}|.
$$
\v

\n Since $E\geq cq^{\frac{3d+1}{4}}(\log{q})^{\gamma/2}\geq cq^{\frac{d+1}{2}}(\log{q})^{\gamma}$, it follows from Lemma \ref{salem_edge} that 
\v

$$
\sum_{x,y\in \mathbb{F}_q^d}E(x)E(y)S(x-y)\geq K\frac{|E|^2}{2q},
$$
\v

\n and therefore by the pigeonhole principle there must exist some $u\in S$, $u\neq \pm v$ so that 
\v

$$
|E\cap (E-u)|=\sum_{\substack{x,y\in \mathbb{F}_q^d \\ x-y=u}}E(x)E(y)\geq \frac{|E|^2}{4q^d}
$$
\v

\n Let $E_u=E\cap (E-u)$, and since 
\v

$$
|E_u|\geq \frac{|E|^2}{4q^d}\geq \frac{1}{4}c^2q^{\frac{d+1}{2}}(\log{q})^{\gamma},
$$
\v

\n we can apply Lemma \ref{salem_edge} once again to find that
\v

$$
\sum_{x,y\in \mathbb{F}_q^2}E_u(x)E_u(y)S(x-y)\geq K\frac{|E_u|^2}{2q}.
$$
\v

\n Therefore, since
$$
|\{(x,y)\in E_u^2: x-y\in \{\pm u, \pm v,\pm u\pm v\}\}|
\leq 8|E_u|<K\frac{|E_u|^2}{2q}\leq |\{(x,y)\in E_u^2: x-y\in S\}|,
$$
\v

\n we find that there must exist some pair $(x,y)\in E_u^2$ such that $x-y\in S$ but $x-y\notin\{\pm u,  \pm v\}$, since the second inequality above is strict.  Let $x^1=x+u$, $x^2=y+u$, $x^3=x$, $x^4=y$, and this completes the proof.
\end{proof}
\v

\n In the course of proving Lemma \ref{rhombus}, we have also obtained the following fact:
\v

\begin{corollary}\label{pigeon}
If $|E|\geq cq^{\frac{d+1}{2}}(\log{q})^{\gamma}$, $c$ sufficiently large, it follows that there exists $v\in S$ satisfying
$$
E_v\geq K\frac{|E|^2}{4q^d}.
$$
\end{corollary}
\v

\n This fact, along with Lemma \ref{rhombus}, allows us to construct a cube graph which contains most of the features necessary for shattering 3 points.  If $|E|\geq Cq^{\frac{7d+1}{8}}(\log{q})^{\gamma/4}$, then
\v

$$
|E_v|\geq \frac{|E|^2}{4q^d}\geq \frac{1}{4}c^2q^{\frac{3d+1}{4}}(\log{q})^{\gamma/2},
$$
\v

\n and so we can apply Lemma \ref{rhombus} to the set $E_v$ instead of $E$, using the vector $v$ chosen by Corollary \ref{pigeon}.  
\v

\begin{definition}
Let $Q_r$ be the $r$-dimensional cube graph, i.e. the graph with vertex set $V=\{0,1\}^r$, with an edge
$$
(a_1,\dots,a_r)\sim (b_1,\dots,b_r)
$$
if and only if $a_i\neq b_i$ for exactly one index $i$.
    
\end{definition}
\v

\begin{corollary}\label{cube}
Suppose $S$ is symmetric, i.e. $S=-S$, so that $\mathcal{G}_S(E)$ is an undirected graph.  If $E\subseteq \mathbb{F}_q^d$, $|E|\geq cq^{\frac{7d+1}{8}}(\log{q})^{\gamma/4}$, $c$ sufficiently large, then there is a subgraph of $\mathcal{G}_S(E)$ isomorphic to the graph obtained from $Q_3$ by deleting one vertex.
\end{corollary}
\v

\begin{proof}
As mentioned, we combine Lemma \ref{rhombus} and Corollary \ref{pigeon}.  With $v,x^1,x^2,x^3,x^4$ defined as above, the proof is the picture in Figure \ref{cube_graph}.  The fact that Lemma \ref{rhombus} ensures $x^i-x^j$ is not equal to $\pm v$ for any $(i,j)$, and that $u\neq \pm v$, ensures that all the vertices are actually distinct.
    
\end{proof}

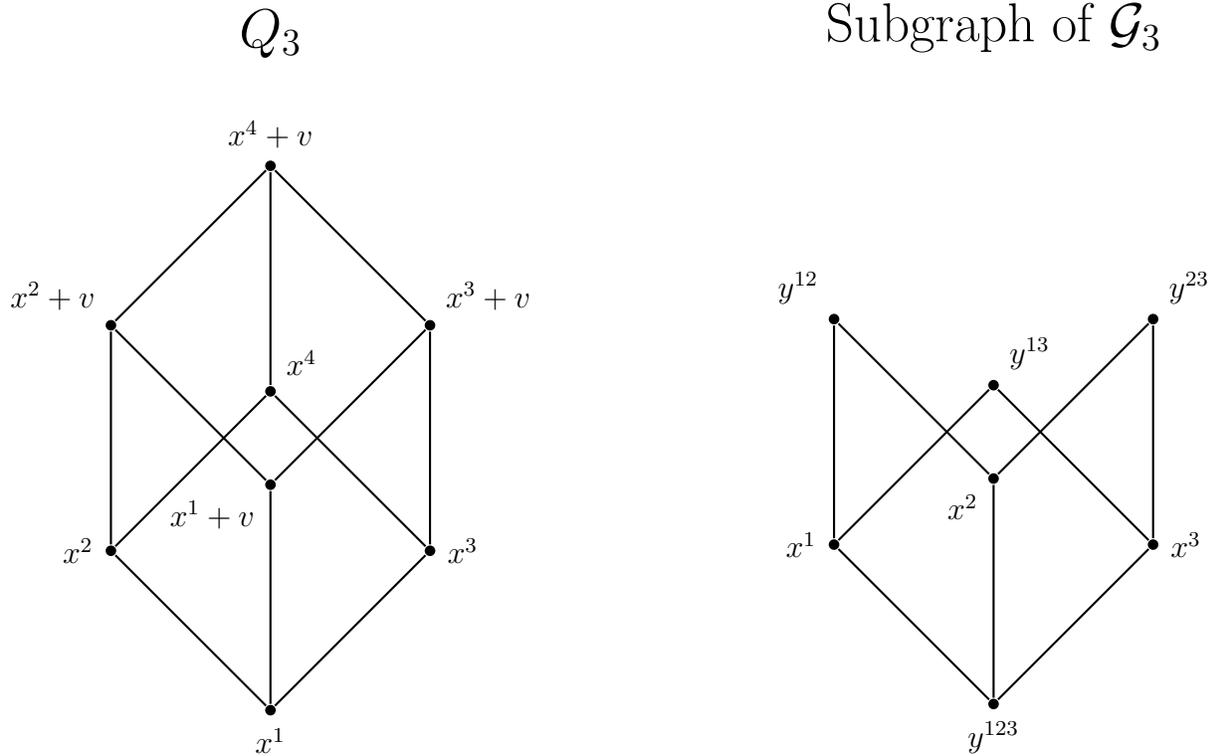
\begin{figure}[h!]\label{cube_graph}
\centering

\begin{minipage}[t]{.4\textwidth}
\centering
\begin{tikzpicture}[scale=3, thick]
  \node[fill=black, circle, minimum size=4pt, inner sep=0pt, label=below:$x^1$] (y123) at (0,0) {};
  \node[fill=black, circle, minimum size=4pt, inner sep=0pt, label=below left:$x^1+v$] (x2) at (0,1) {};
  \node[fill=black, circle, minimum size=4pt, inner sep=0pt, label=left:$x^2$] (x1) at (-.707,.707) {};
  \node[fill=black, circle, minimum size=4pt, inner sep=0pt, label= right:$x^3$] (x3) at (.707,.707) {};
  \node[fill=black, circle, minimum size=4pt, inner sep=0pt, label=above left:$x^2+v$] (y12) at (-.707,1.707) {};
  \node[fill=black, circle, minimum size=4pt, inner sep=0pt, label=above right:$x^3+v$] (y23) at (.707,1.707) {};
  \node[fill=black, circle, minimum size=4pt, inner sep=0pt, label=above right:$x^4$] (y13) at (0,1.414) {};
  \node[fill=black, circle, minimum size=4pt, inner sep=0pt, label=above:$x^4+v$] (z) at (0,2.414) {};
  
  \node at (0,3) {\LARGE $Q_3$};

  \draw (y123) -- (x1);
  \draw (y123) -- (x2);
  \draw (y123) -- (x3);
  \draw (y12) -- (x1);
  \draw (y12) -- (x2);
    \draw (y13) -- (x1);
    \draw (y13) -- (x3);
    \draw (y23) -- (x2);
    \draw (y23) -- (x3);
    \draw (z) -- (y13);
    \draw (z) -- (y12);
    \draw (z) -- (y23);
  
\end{tikzpicture}

\end{minipage}
\hfill
\begin{minipage}[t]{.4\textwidth}
\centering
\begin{tikzpicture}[scale=3, thick]
  \node[fill=black, circle, minimum size=4pt, inner sep=0pt, label=below:$y^{123}$] (y123) at (0,0) {};
  \node[fill=black, circle, minimum size=4pt, inner sep=0pt, label=below left:$x^2$] (x2) at (0,1) {};
  \node[fill=black, circle, minimum size=4pt, inner sep=0pt, label=left:$x^1$] (x1) at (-.707,.707) {};
  \node[fill=black, circle, minimum size=4pt, inner sep=0pt, label= right:$x^3$] (x3) at (.707,.707) {};
  \node[fill=black, circle, minimum size=4pt, inner sep=0pt, label=above left:$y^{12}$] (y12) at (-.707,1.707) {};
  \node[fill=black, circle, minimum size=4pt, inner sep=0pt, label=above right:$y^{23}$] (y23) at (.707,1.707) {};
  \node[fill=black, circle, minimum size=4pt, inner sep=0pt, label=above right:$y^{13}$] (y13) at (0,1.414) {};

  \node at (0,3) {\LARGE Subgraph of $\mathcal{G}_3$};

  \draw (y123) -- (x1);
  \draw (y123) -- (x2);
  \draw (y123) -- (x3);
  \draw (y12) -- (x1);
  \draw (y12) -- (x2);
    \draw (y13) -- (x1);
    \draw (y13) -- (x3);
    \draw (y23) -- (x2);
    \draw (y23) -- (x3);

\end{tikzpicture}

\end{minipage}

\caption{Left: The cube graph $Q_3$, as constructed by Lemma \ref{rhombus} and Corollary \ref{cube}. Right: Re-labeling the vertices to agree with Figure \ref{vc3_graph}, and deleting the vertex $x^4+v$ which is unnecessary for the proof of Theorem \ref{shatter3}.}
\end{figure}
\v

\n We are almost ready to prove Theorem \ref{shatter3}, but first we need Lemma \ref{functional}, which will allow us to estimate the number of triples $(x,y,z)\in E^3$ with $x-y,z-y\in S$, as well as Lemma \ref{pruning}, which will allow us to assume without loss of generality that each of the vertices constructed above has a positive proportion of the average number of neighbors for vertices in $\mathcal{G}_S(E)$.
\v

\begin{lemma}\label{functional}
Let $f,g:\mathbb{F}_q^d\to \mathbb{R}^{\geq 0}$, and let $S\subseteq \mathbb{F}_q^d$ satisfy $|\hat{S}(m)|\lesssim q^{-\frac{d+1}{2}}(\log{q})^{\gamma}$ for $m\neq 0$, and $|S|=Kq^{d-1}(1+O(q^{-\alpha}))$.  It follows that
\v

$$
\sum_{x,y\in \mathbb{F}_q^d}f(x)g(y)S(x-y)
=\frac{K}{q}||f||_{L^1}||g||_{L^1}(1+O(q^{-\alpha}))+O\left(q^{\frac{d-1}{2}}(\log{q})^{\gamma}||f||_{L^2}||g||_{L^2}\right).
$$
    
\end{lemma}
\v

\begin{proof}
$$
\sum_{x,y\in \mathbb{F}_q^d}f(x)g(y)S(x-y)
=\sum_{x,y,m\in \mathbb{F}_q^d}f(x)g(y)\chi(m\cdot (x-y))\hat{S}(m)
=q^{2d}\sum_m\overline{\hat{f}(m)}\hat{g}(m)\hat{S}(m)
$$
$$
=\frac{K}{q}||f||_{L^1}||g||_{L^1}(1+O(q^{-\alpha}))+\mathcal{R},
$$
where
\v

$$
\mathcal{R}=q^{2d}\sum_{m\neq 0}\overline{\hat{f}(m)}\hat{g}(m)\hat{S}(m)
\lesssim q^{\frac{3d-1}{2}}(\log{q})^{\gamma}\sum_{m\in \mathbb{F}_q^2}|\hat{f}(m)||\hat{g}(m)|
\leq q^{\frac{3d-1}{2}}(\log{q})^{\gamma}||\hat{f}||_{L^2}||\hat{g}||_{L^2}
$$
\v

$$
=q^{\frac{d-1}{2}}(\log{q})^{\gamma}||f||_{L^2}||g||_{L^2}
$$
    
\end{proof}
\v

\begin{lemma}\label{pruning}

Let $E\subseteq \mathbb{F}_q^d$, $|E|\geq cq^{\frac{d+1}{2}}(\log{q})^{\gamma}$ for $c$ sufficiently large.  Suppose that $S$ satisfies the conditions of Lemma \ref{salem_edge}, i.e. $|\hat{S}(m)|\lesssim q^{-\frac{d+1}{2}}(\log{q})^{\gamma}$ for $m\neq 0$ and $|S|=Kq^{d-1}(1+O(q^{-\alpha}))$.  Then whenever $0<M\lesssim q^{\beta}$, $\beta<\frac{d-1}{2}$, there exists a subset $E_M\subseteq E$, $|E_M|\gtrsim |E|$, so that 
$$
E\ast S(x)> M
$$
for all $x\in E_M$.

\end{lemma}
\v

\begin{proof}
Summing over $x\in E$, we have
\v

$$
\sum_{x\in E}E\ast S(x)
=\sum_{x,y\in \mathbb{F}_q^d}E(x)E(y)S(x-y)
\geq K\frac{|E|^2}{2q}.
$$
Let
\v

$$
E_M=\left\{x\in E: E\ast S(x)> M\right\}.
$$
\v

\n Then, since $M\leq \frac{K}{4}|E|q^{-1}$ for $q$ sufficiently large, it follows that
\v

$$
\sum_{x\in E}E\ast S(x)\leq M|E\sm E_M|+\sum_{x\in E_M}E\ast S(x)
\leq K\frac{|E|^2}{4q}+\sum_{x\in E_M}E\ast S(x).
$$
\v

\n Therefore,
$$
\sum_{x\in E_M}E\ast S(x)\geq K\frac{|E|^2}{4q}.
$$
\v

\n By Cauchy-Schwarz,
$$
K^2\frac{|E|^4}{16q^2}\leq 
\left(\sum_{x\in E_M}E\ast S(x)\right)^2
\leq |E_M|\sum_{x\in E}(E\ast S(x))^2.
$$
\v

\n We can bound this last sum using Lemma \ref{functional} with $f(x)=E(x)E\ast S(x)$, $g(y)=E(y)$.  We have
\v

$$
\sum_{x\in E}(E\ast S(x))^2
=\sum_{x,y\in \mathbb{F}_q^d}E(x)E\ast S(x)E(y)S(x-y)
=\sum_{x,y}f(x)g(y)S(x-y)
$$
\v

$$
=\frac{K}{q}||f||_{L^1}||g||_{L^1}(1+O(q^{-\alpha}))+O\left(q^{\frac{d-1}{2}}(\log{q})^{\gamma}||f||_{L^2}||g||_{L^2}\right)
$$
\v

\n Clearly, $||g||_{L^1}=|E|$ and $||g||_{L^2}=|E|^{\frac{1}{2}}$.  Moreover,
\v

$$
||f||_{L^1}=\sum_{x}E(x)E\ast S(x)
=\sum_{x,y}E(x)E(y)S(x-y)\lesssim K\frac{|E|^2}{q},
$$
and
$$
||f||_{L^2}^2=\sum_{x\in E}(E\ast S(x))^2
$$
\v

\n is the original sum we were trying to bound.  We see that
\v

$$
\frac{K}{q}||f||_{L^1}||g||_{L^1}(1+O(q^{-\alpha}))\lesssim K\frac{|E|^3}{q^2}.  
$$
\v

\n If this is the larger of the two terms, we are done.  Otherwise, we have
\v

$$
\sum_{x\in E}(E\ast S(x))^2
=||f||_{L^2}^2\lesssim q^{\frac{d-1}{2}}(\log{q})^{\gamma}||f||_{L^2}||g||_{L^2},
$$
\v

\n and so
$$
||f||_{L^2}\lesssim q^{\frac{d-1}{2}}(\log{q})^{\gamma}|E|^{\frac{1}{2}}.
$$
\v

\n Finally, we obtain
\v

$$
\sum_{x\in E}(E\ast S(x))^2\lesssim q^{d-1}(\log{q})^{2\gamma}|E|\lesssim \frac{|E|^3}{q^2},
$$
\v

\n since $|E|\geq cq^{\frac{d+1}{2}}(\log{q})^{\gamma}$.  Putting this all together, we get the desired conclusion that $|E_M|\gtrsim |E|$.

\end{proof}
\v

\n We are now ready to prove Theorem \ref{shatter3}.  
\v

\begin{proof}[Proof of Theorem \ref{shatter3}]

To begin, using the labeling conventions from Figure \ref{cube_graph}, we simply ignore the vertex $x^4+v$, and relabel the rest of the vertices to agree with the labelling of Figure \ref{vc3_graph}. It only remains to find points $y^1,y^2,y^3,y^{\es}$ with the desired connections.  $y^{\es}$ is trivial, since it is not connected to any other vertices of $\mathcal{G}_3$, so we need only find a point in $E$ which is not connected to any of the previously constructed points.  But there are finitely many such points (7 to be precise) and each has $O(q)$ neighbors in $\mathcal{G}_S(E)$.  Since $E\geq cq^{\frac{7d+1}{8}}$, clearly we can pick a vertex which is not connected to any of these 7 points in $\mathcal{G}_S(E)$.  
\\

\n Next we need to find $y^1,y^2,y^3\in E$ so that $x^i-y^j\in S$ if and only if $i=j$.  Let $E_M$ be as in Lemma \ref{pruning}, and note that we may assume without loss of generality that $x^1,x^2,x^3\in E_M$.  This is because $|E_M|\gtrsim |E|$, so $E_M$ satisfies the same assumptions as $E$, and thus all of these constructions apply to $E_M$ as well as $E$.  
\\

\n Thus, we have guaranteed that each of $x^1,x^2,x^3$ have more than $M$ neighbors in $\mathcal{G}_S(E)$.  Therefore we can choose an appropriate $y^1$ as long as $M$ is large enough to ensure that $x^1$ has a neighbor which avoids the 7 points already in the diagram, and also avoids any points which are neighbors of $x^2$ or $x^3$.  But, since $|S\cap (S-x)|\leq q^{\beta}$ for all $x\in \mathbb{F}_q^2$, there are only $7+2q^{\beta}$ points which must be avoided.  Put $M=7+2q^{\beta}\lesssim q^{\beta}$, and thus we can find $y^1,y^2,y^3$ with the desired connections in $\mathcal{G}_S(E)$.  This completes the proof.

\end{proof}
\v

\n In addition, we now have the Fourier analytic machinery needed to prove Lemma \ref{shatter2}.
\v

\begin{proof}[Proof of Lemma \ref{shatter2}]

We need to find
\v

$$
\{x^1,x^2\}\cup\{y^{12},y^{1},y^2,y^{\es}\}\subseteq E
$$
\v

\n such that $x^i\in y^I$ if and only if $i\in I$.  It is trivial to find $y^{\es}$ once the other points are constructed, since we have to avoid $O(q)$ points which is significantly less than the size of $E$.  Therefore, we really just need to achieve the rest of the configuration as displayed in Figure \ref{path4}.
\v

\n In the proof of Lemma \ref{pruning}, we use Lemma \ref{functional} to find that 
\v

$$
\sum_{x\in E}(E\ast S(x))^2\lesssim \frac{|E|^3}{q^2}
$$
\v

\n as long as $|E|\geq cq^{\frac{d+1}{2}}(\log{q})^{\gamma}$ for $c$ sufficiently large.  In fact, that same calculation yields the more precise estimate
\v

$$
\sum_{x\in E}(E\ast S(x))^2=K^2\frac{|E|^3}{q^2}(1+O(q^{-\alpha}))+O(q^{d-1}(\log{q})^{2\gamma}|E|),
$$
\v

\n where $O(q^{d-1}(\log{q})^{2\gamma}|E|)$ is a small error term.  But the sum $\sum_{x\in E}(E\ast S(x))^2$ is counting the number of tuples $(x^1,x^2,y^{12})\in E^3$ with $x^1-y^{12},x^2-y^{12}\in S$.  Therefore, we may apply Lemma \ref{pruning} once again to ensure that $(x^1,x^2,y^{12})\in E_M\subseteq E$, $M>q^{\beta}$.  This means that the number of points $y^1\in E$ with $x^1-y^1\in S$ is greater than the number of points in the intersection
$$
|(x^1-S)\cap (x^2-S)|\leq q^{\beta},
$$
and therefore we may choose a point $y^1\in E$ satisfying the shattering configuration property that $x^1-y^1\in S$ and $x^1-y^2\notin S$.  We may similarly choose an appropriate $y^2$, and this completes the proof of the Lemma.

\end{proof}
\v

\section{Symmetrized Parabola}
\v

\n When first considering the question of how the results of \cite{FIMW23} generalize to $\mathcal{H}_S(E)$ where $S$ is a general quadratic curve as opposed to a circle, it did not take long to realize that the argument does not apply to a parabola because of a lack of symmetry.  In attempts to fix this by considering a symmetrized parabola, we then lose the necessary intersection properties, since a symmetrized parabola may intersect a translation of itself up to 6 times over a finite field (only 4 times in Euclidean space), which is more than the 2 points of intersection obtained by translating a circle.  In this section we demonstrate how the situation over finite fields is fundamentally different from Euclidean space, in contrast with the situation for circles and spheres, in which the VC-dimension is the same in $\mathbb{F}_q^d$ and in $\mathbb{R}^d$.  We will show that one can only shatter 3 points with the symmetrized parabola over Euclidean space, while it is possible to shatter 4 points over (some) finite fields.  
\v

\n This leads to a very natural follow-up question: In \cite{IMMM25}, the VC-dimension of a hypothesis class of spheres in fractal sets in $\mathbb{R}^d$ is studied.  In particular, they show that whenever the Hausdorff dimension of $E\subseteq \mathbb{R}^d$, $d\geq 3$, is above a certain threshold, it follows that $\text{VCdim}(\mathcal{H}_{S_t}(E))\geq 3$, where $\mathcal{H}_{S_t}(E)$ is the hypothesis class defined by translations of the sphere $S_t$ of radius $t$ intersected with $E$.  The methods do not provide a result in $d=2$, and indeed there is a fundamental obstruction due to Maga's construction \cite{M10} of full-dimensional sets avoiding parallelograms.  There may be no hope to extend the results of \cite{IMMM25} to $d=2$, but what about for the symmetrized parabola, or other curves besides the circle?  It is plausible that one could obtain a positive result in this case, but simply applying the techniques from \cite{IMMM25} is not sufficient.  For now we leave this as an open problem.
\v

\n Here we will discuss the VC-dimension of a hypothesis class defined by a symmetrized parabola, and compare the finite field situation with Euclidean space.
\v

\begin{definition}
Let $P=P^+\cup P^-$ be the symmetrized parabola, where $P^+=\{(x,x^2)\}$ and $P^-=-P^+$.  Let $L,R$ be the left and right halves of the symmetrized parabola respectively, and let
\v

$$
L^+=P^+\cap L, \ \ \ \ \  L^-=P^-\cap L, \ \ \ \ \ R^+=P^+\cap R, \ \ \ \ \ R^{-}=P^-\cap R
$$

\end{definition}
\v

\begin{remark}
The definition for $P^+$ and $P^-$ can be interpreted over any ring, but we will consider the cases $\mathbb{R}^2$ and $\mathbb{F}_q^2$ here.  The definition for $L,R,L^{\pm},R^{\pm}$ can be interpreted over any ordered ring.  In particular this excludes $\mathbb{F}_q$.
\v
   
\end{remark}
\v

\subsection{Euclidean space}
\v

\begin{definition}
\n Consider the symmetrized parabola in Euclidean space.  Let $\mathcal{H}_P(\R^2)$ denote the hypothesis class consisting of indicator functions $h_y$ of sets of the form $y+P$ for $y\in \mathbb{R}^2$.
    
\end{definition}
\v

\begin{proposition}\label{shatter4}
$\mathcal{H}_P(\R^2)$ has VC-dimension equal to 3.

\end{proposition}
\v

\n First we will show that no set of 4 points, $\{x^1,x^2,x^3,x^4\}$, is shattered by $\mathcal{H}_P(\R^2)$.  For sake of contradiction, assume that we have a shattering configuration
\v

$$
\{x^1,x^2,x^3,x^4\}\cup \{y^I: I\subseteq[4]\}
$$
\v

\n such that $x^i\in y^I$ if and only if $i\in I$. Without loss of generality, assume $y^{1234}=0$.  We observe that there are several cases based on the position of the points $x^1,\dots,x^4$ with respect to the 4 quadrants of the plane.  
\v

\begin{lemma}\label{left_unique}
Consider the points $X=(x,x^2)\in L^+$ and $Y=(y,-y^2)\in L^-$, so that $x,y<0$.   
If there exists $P_0 \in \mathbb{R}^2$ such that $X\in P_0+L^+$ and $Y\in P_0+L^-$,
then $P_0 = (0,0)$
\end{lemma}
\v

\begin{proof}
Let $(a,b)=(y-x,y^2-x^2)$, and the Lemma is equivalent to showing that
$$
\left(\begin{array}{c}
x+a \\
x^2+b
\end{array}\right)
=\left(\begin{array}{c}
y \\
-y^2
\end{array}\right)
$$
has at most one solution $(x,y)$ with $x,y<0$. But this is the intersection of the circle $x^2+y^2=-b$ and the line $y=x+a$, which is at most 2 points.  One of these points has $x>0$, so there is only one solution.
\end{proof}
\v

\begin{lemma}\label{s_unique}
Let $D=L^-\cup R^+$.  Then, for any two points $P$ and $Q$, there are at most 2 solutions $v\in \mathbb{R}^2$ satisfying $P,Q\in (D+v)$.  The same holds for $-D$ by symmetry.

\end{lemma}
\v

\begin{proof}

Let
$$
f(x)=\left\{\begin{array}{ll}
x^2 & :x\geq 0 \\
-x^2 & :x<0
\end{array}\right.
$$
\v

\n We are looking for solutions to $g(x):=f(x+h)-f(x)=k$, where $Q-P=(h,k)$. There is a one-to-one correspondence between solutions to this equation and translations satisfying $P,Q\in (D+v)$, by taking $v=P-(x_0,f(x_0))$.  We compute the derivative
$$
g'(x)=\left\{\begin{array}{ll}
2h & :x\geq 0 \\
-2h & :x\leq -h \\
4x+2h & :-h<x<0
\end{array}\right.
$$
Therefore, $g$ is decreasing to the left of $x=-\frac{h}{2}$ and increasing to the right of $x=-\frac{h}{2}$, so we conclude that there are at most 2  solutions to the equation $g(x)=k$.  

\end{proof}














\begin{figure}
    \centering
\begin{tikzpicture}[scale=0.8]
  \draw[domain=-2:2,smooth,variable=\x,black,thick] plot ({\x},{\x*\x});
  \draw[domain=-2:2,smooth,variable=\x,black,thick] plot ({\x},{-\x*\x});
  
  \node[fill=black, circle, minimum size=4pt, inner sep=0pt, label=left:$x^{1}$] (x_1) at (-1.8708,3.5) {};

    \node[fill=black, circle, minimum size=4pt, inner sep=0pt, label= left:$x^{2}$] (x_2) at (1.6124,2.5) {};

    \node[fill=black, circle, minimum size=4pt, inner sep=0pt, label= left:$x^{4}$] (x_4) at (-1.2247,-1.5) {};

    \node[fill=black, circle, minimum size=4pt, inner sep=0pt, label= left:$x^{3}$] (x_3) at (-0.7071,-0.5) {};

\end{tikzpicture}
    \caption{Four points placed on the Symmetrized Parabola, used to illustrate the different geometric configurations examined in the casework that follows}
\end{figure}
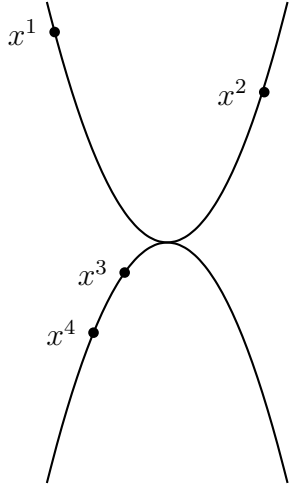

\begin{lemma}\label{height_of_points}
    Assume that $\{x^1,\dots, x^4\}$ are shattered, with the notation from above for $x^i$ and $y^I$, where $x^i=(x_1^i,x_2^i)\in \mathbb{R}^2$. Then $x_2^1,x_2^2,x_2^3,x_2^4$ are distinct.
\end{lemma}

\begin{proof}
    Suppose, for contradiction, that there exists distinct points $x^i \neq x^j$ such that $x^i_2 =x^j_2$. Without loss of generality, assume $x^i,x^j \in y^{1234}+P^+$, since they must either be both in the top branch or both in the bottom branch. Since $(y^{1234}+P^+)\cap (y^{123}+P^+)$ has at most one point, we see that $x^i$ and $x^j$ cannot both lie in $y^{123}+P^+$. Moreover, since they share the same height, it is impossible for one to lie on $y^{123}+P^+$ and the other on $y^{123}+P^-$. It follows that $x^i,x^j \in y^{123}+P^-$. By a similar argument, we must also have $x^i,x^j \in y^{124}+P^-$. This implies $y^{123}=y^{124}$, a contradiction. 
\end{proof}

\begin{lemma}\label{shatter4Upper}
    Assume that $\{x^1,\dots,x^4\}$ are shattered, with the notation from above for $x^i$ and $y^I$, where $x^i=(x_1^i,x_2^i)\in \mathbb{R}^2$. Further assume that the points are indexed so that their heights are in strictly decreasing order, i.e. $x_2^1 > x_2^2 > x_2^3 > x_2^4$. Then $x^1,x^2,x^3,x^4$ cannot all lie in $P^+.$
\end{lemma}

\begin{proof}
    Suppose $x^1,\dots,x^4 \in P^+$. It follows that $x^3 ,x^4\in y^{134} + P^-$, since $x^1\in P^+\cap (y^{134}+P^+),$ and the intersection contains at most one point.  For similar reasons $x^3,x^4 \in y^{234}+P^-$. This implies $y^{134}=y^{234}$, a contradiction.
\end{proof}

\begin{lemma}\label{x3upper}
     Assume that $\{x^1,\dots,x^4\}$ are shattered, with the notation from above for $x^i$ and $y^I$, where $x^i=(x_1^i,x_2^i)\in \mathbb{R}^2$. Further assume that the points are indexed so that their heights are in strictly decreasing order, i.e. $x_2^1 > x_2^2 > x_2^3 > x_2^4$. Then $x^3$ cannot lie in $P^+.$   
\end{lemma}
\begin{proof}
    Suppose $x^3 \in P^+$. It follows that $x^1,x^2 \in P^+$. Moreover, following lemma $\ref{shatter4Upper}$ $x^4\in P^-$. By the same reasoning as before, $x^3$ must simultaneously satisfy 

    \begin{align}
        x^3 \in y^{134}+P^- \quad\text{and}\quad x^3\in y^{234}+P^-
    \end{align}
    However, clearly $x^4\in (y^{134}+P^-)\cap (y^{234}+P^-)$ as well, forcing $y^{134}=y^{234}$, a contradiction.
\end{proof}

\begin{lemma}\label{main_case}
 Suppose that
 $$
 \{x^1,x^2,x^3,x^4\}\cup\{y^I: I\subseteq [4]\}
 $$
 is a shattering configuration with $x^i=(x_1^i,x_2^i)$, and $x_2^1>x_2^2>x_2^3>x_2^4$.  Then, we must have $x^1,x^2\in L^+$, $x^3\in L^-$, $x^4\in R^-$.    
\end{lemma}
\v

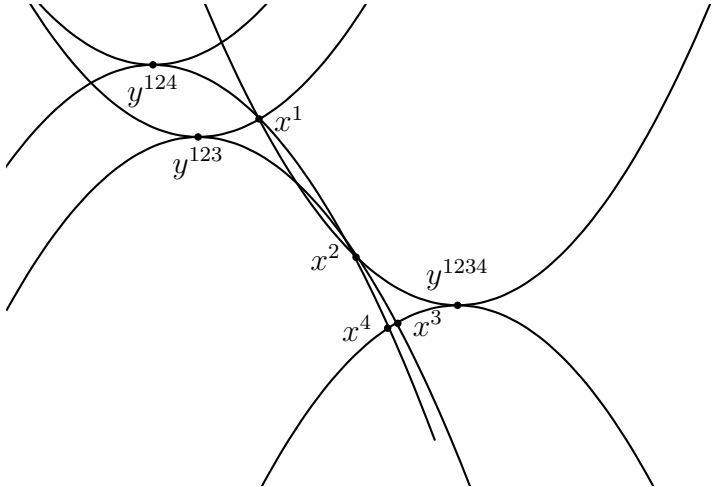
\begin{figure}[h!]
\centering
\begin{tikzpicture}[xscale=1.5, yscale=0.8] 
\clip (-4,-3) rectangle (5,5);
  \draw[domain=-2.3:2.3,smooth,variable=\x,thick] plot ({\x},{\x*\x});
  \draw[domain=-2.3:2.3,smooth,variable=\x,thick] plot ({\x},{-\x*\x});

  \node[fill=black, circle, inner sep=1pt, label=right:$x^{1}$] at (-1.76, 3.1) {};
  \node[fill=black, circle, inner sep=1pt, label=left:$x^{2}$] at (-0.9, 0.8) {};
  \node[fill=black, circle, inner sep=1pt, label=right:$x^{3}$] at (-0.53,-0.3) {};
  \node[fill=black, circle, inner sep=1pt, label=left:$x^{4}$] at (-0.62,-0.38) {};
  \node[fill=black, circle, inner sep=1pt, label=above:$y^{1234}$] at (0,0) {};
  \node[fill=black, circle, inner sep=1pt, ] at (-2.3, 2.8) {};
  \node[fill=black, circle, inner sep=1pt, ] at (-2.7, 4) {};

  \def\aone{-2.3}
  \def\bone{2.8}
  \draw[domain=-4:0.5,smooth,variable=\x,thick] plot ({\x},{-(\x-\aone)^2 + \bone});
  \draw[domain=-4:0.5,smooth,variable=\x,thick] plot ({\x},{(\x-\aone)^2 + \bone});
  \fill (\aone,\bone) circle (1pt) node[below] {$y^{123}$};

  \def\atwo{-2.7}
  \def\btwo{4}
  \draw[domain=-4.5:-0.2,smooth,variable=\x,thick] plot ({\x},{-(\x-\atwo)^2 + \btwo});
  \draw[domain=-4.5:-0.2,smooth,variable=\x,thick] plot ({\x},{(\x-\atwo)^2 + \btwo});
  \fill (\atwo,\btwo) circle (1pt) node[below] {$y^{124}$};

\end{tikzpicture}
\caption{
Geometric configuration used in proof of Lemma \ref{main_case}, Case 1, where the translates $y^{123}+P$ and $y^{124}+P$ force $x^3,x^4$ to lie to the right of $x^1,x^2$.
}
\end{figure}

\begin{proof}
 \n   Suppose $x^3 \in P^-$. By Lemma~\ref{x3upper}, we know $x^1,x^2\in P^+$. We divide into cases based on the number of points in $L^+$ and $L^-$. 
 \v

\n \textbf{Case 1:} $x^1,x^2 \in L^+$ and $x^3,x^4 \in L^-$.
Since $(y^{1234}+P^+)\cap(y^{123}+P^+)$ contains at most one point, $x^1$ and $x^2$ cannot both lie on $y^{123}+P^+$.
Thus, $x^2$ must lie on $y^{123}+P^-$.
If $x^2 \in y^{123}+L^-$, then so is $x^3$.  This also forces $x^1\in y^{123}+L^+$.  By Lemma \ref{left_unique}, since $x^1\in L^+$ and $x^3\in L^-$, we cannot also have $x^1\in y^{123}+L^+$ and $x^3\in y^{123}+L^-$.  Therefore, we conclude that $x^2\in y^{123}+R^-$, and therefore $x^3\in y^{123}+R^-$ as well.
\v

\n A symmetric argument applied to $y^{124}$ gives $x^2,x^4 \in y^{124}+R^-$.
Hence, $x^3$ and $x^4$ both lie to the right of $x^1,x^2$.
However, by reflecting the entire picture across the horizontal axis we see that the same argument could have been applied to show that $x^3$ and $x^4$ both lie to the left of $x^1$ and $x^2$.  This is a contradiction.  Therefore, it is impossible to have $x^1,x^2 \in L^+$ and $x^3,x^4 \in L^-$.
\v

\n \textbf{Case 2:} $x^1,x^2 \in L^+$ and $x^3,x^4\in R^-$. Since $(y^{1234}+P^+)\cap(y^{123}+P^+)$ contains at most one point, $x^1$ and $x^2$ cannot both lie on $y^{123}+P^+$. Hence, at least one of them must lie on $y^{123}+P^-$. Furthermore, in order to have $x^3\in y^{123}+P$, we must have $x^1,x^2 \in y^{123} + L$, since otherwise we would have $y^{123}$ lying entirely to the left of $P$, in which case $y^{123}+P$ would not intersect $R^-$. It is impossible for $x^1\in y^{123}+L^-$ for then, $x^2$ could not be hit. Thus, we must have $x^1\in y^{123}+L^+$ and $x^2\in y^{123}+L^-$. A symmetric argument with $y^{124}$ gives $x^1\in y^{124}+L^+$ and $x^2\in y^{124}+L^-$. However, by Lemma $\ref{left_unique}$ this implies $y^{123}=y^{124}$, a contradiction.
\v

\n  \textbf{Case 3:} $x^1,x^2\in L^+$, $x^3\in R^-$, $x^4\in L^-$.  Now, since $(y^{1234}+P^-)\cap (y^{234}+P^-)$ contains at most one point, $x^3$ and $x^4$ cannot both lie on $y^{234}+P^-$, and so $x^3\in y^{234}+P^+$, and so is $x^2$.  We see that $x^2\in y^{234}+L^+$, since otherwise $y^{234}$ would lie entirely to the left of $P$, contradicting that $x^3\in y^{234}+P$.  We see that $x^3\in y^{234}+R^+$ for similar reasons.  When we consider the position of $x^4$, we know by Lemma \ref{left_unique} it cannot be in $y^{234}+L^-$ since we already have $x^2\in L^+$ and $x^4\in L^-$.  We cannot have $x^4\in y^{234}+R$, since then $y^{234}$ would lie entirely to the left of $P$.  Thus, we conclude that $x^4\in y^{234}+L^+$.  The same argument applies to $y^{134}$, and so $\{x^1,x^3,x^4\}\subseteq y^{134}+P^+$.  This yields a contradiction, since $(y^{134}+P^+)\cap (y^{234}+P^+)$ has at most one point and cannot contain both $x^3$ and $x^4$.
\v

\n \textbf{Case 4:} $x^1\in L^+,x^2\in R^+,x^3\in L^-,x^4\in R^-$. Since $(y^{1234}+P^+)\cap(y^{123}+P^+)$ contains at most one point, $x^1$ and $x^2$ cannot both lie on $y^{123}+P^+$.  Therefore, $x^2,x^3\in y^{123}+P^-$.  Moreover, $x^1$ is also in $y^{123}+P^-$ because it cannot be in $y^{123}+L^+$ because of Lemma \ref{left_unique} applied to $\{x^1,x^3\}\subseteq L$, and it cannot be in $y^{123}+R$ because then $y^{123}$ would lie entirely to the left of $P$, contradicting that $x^2\in y^{123}+P$.  Now considering $y^{124}$, note that $x^1$ and $x^2$ cannot both lie on $y^{124}+P^+$, since $x^1,x^2\in P^+$, and also cannot both lie on $y^{124}+P^-$, since $x^1,x^2\in y^{123}+P^-$.  Therefore, $x^1\in y^{124}+P^+$ and $x^2\in y^{124}+P^-$.  We also see that $x^2,x^4\in y^{124}+R^-$, since if either were in a left branch then $y^{124}$ would lie entirely to the right of $P$, contradicting that $x^1\in y^{124}+P$.  For similar reasons, $x^1\in y^{124}+L^+$.   
\v

\n Now we consider $y^{234}$.  Since $\{x^3,x^4\}\subseteq P^-$, they cannot both be contained in $y^{234}+P^-$, so we see that $x^3\in y^{234}+P^+$ and hence so is $x^2$.  Once again we conclude that in fact $x^3\in y^{234}+L^+$, $x^2\in y^{234}+R^+$, and $x^4\in y^{234}+R$, because $y^{234}$ cannot lie either entirely to the left or entirely to the right of $P$.  Using Lemma \ref{left_unique}, we see that $x^4\in y^{234}+R^+$, since if it were in $y^{234}+R^-$ this would contradict the fact that $x^2\in R^+$ and $x^4\in R^-$.  
\v

\n Finally, consider $y^{134}$.  Since $\{x^3,x^4\}\subseteq P^-$, they cannot both be contained in $y^{134}+P^-$, so we see that $x^1,x^3\in y^{134}+P^+$.  Thus, once again they must be in $y^{134}+L^+$ to avoid $y^{134}$ lying entirely to the left of $P$.  Similarly, $x^3$ and $x^4$ cannot both lie in $y^{134}+P^+$, since they already both lie in $y^{234}+P^+$.  Hence, $x^4\in y^{134}+P^-$, and in fact $x^4\in y^{134}+R^-$ to avoid $y^{134}$ lying entirely to the right of $P$.  By Lemma \ref{s_unique}, we have a contradiction.
\v

\n \textbf{Case 5:}  $x^1\in L^+$, $x^2\in R^+$, $x^3\in R^-$, and $x^4\in L^-$.  Note that $x^2,x^4\in y^{124}+P^-$, since $\{x^1,x^2\}\subseteq P^+$ so they cannot share another upward parabola, and $x^4$ is below $x^2$.  Moreover, $x^2\in y^{124}+R^-$ to avoid $y^{124}$ being to the right of $P$.  Similarly, $x^4\in y^{124}+L^-$ and $x^1\in y^{124}+L$.  We claim $x^1\in y^{124}+L^-$.  If not, then by Lemma \ref{left_unique} we would have $y^{124}=y^{1234}$.  
\v

\n Consider $y^{134}$, observing that $x^1,x^3\in y^{134}+P^+$.  This is because $x^3,x^4\in y^{1234}+P^-$.  In particular, $x^3\in y^{134}+R^+$ because otherwise $y^{134}$ would lie entirely to the right of $P$.  Similarly, we conclude that $x^1\in y^{134}+L^+$ and $x^4\in y^{134}+L$.  We claim $x^4\in y^{134}+L^+$.  If not, then by Lemma \ref{left_unique} we would have $y^{134}=y^{1234}$.  
\v

\n For $y^{234}$, note that $x^2,x^3\in y^{234}+P^+$, since $x^3,x^4\in P^-$.  Moreover, $x^3\in y^{234}+R^+$, to avoid $y^{234}$ lying entirely to the right of $P$.  For similar reasons we conclude $x^2\in y^{234}+R^+$ and $x^4\in y^{234}+L$.  We claim that $x^4\in y^{234}+L^-$.  This is because $x^3,x^4\in y^{134}+P^+$.  
\v

\n Finally, we consider $y^{123}$, observing that $x^2,x^3\in y^{123}+P^-$, since $\{x^1,x^2\}\subseteq P^+$.  Note that $x^2\in y^{123}+R^-$, since otherwise $y^{123}$ would lie entirely to the right of $P$.  Similarly, $x^3\in y^{123}+R^-$.  This yields a contradiction, as $x^2,x^3\in y^{123}+R^-$ implies that $x^3$ is to the right of $x^2$, while $x^2,x^3\in y^{234}+R^+$ implies that $x^2$ is to the right of $x^3$.

\end{proof}

\begin{lemma}\label{dragon}
Suppose that 
$$
\{x^1,x^2,x^3,x^4\}\cup\{y^I: I\subseteq [4]\}
$$
is a shattering configuration, and assume $x_2^1>x_2^2>x_2^3>x_2^4$, $y^{1234}=(0,0)$.  It is not possible to have
$$
x^1,x^2\in L^+, \ \ \ x^3\in L^-, \ \ \ x^4\in R^-.
$$
    
\end{lemma}
\v
\begin{figure}[h!]
\centering
\begin{tikzpicture}[xscale=1.5, yscale=0.8] 
\clip (-4,-3) rectangle (5,5);
  \def\azero{-0.05}
  \def\bzero{0.1}
  \draw[domain=-2.3:2.3,smooth,variable=\x,thick] plot ({\x},{(\x-\azero)^2+\bzero)});
  \draw[domain=-2.3:2.3,smooth,variable=\x,thick] plot ({\x},{-(\x-\azero)^2+\bzero)});

  \node[fill=black, circle, inner sep=1pt, label=right:$x^{1}$] at (-1.818, 3.24) {};
  \node[fill=black, circle, inner sep=1pt, label=left:\hspace{-1cm} $x^{2}$] at (-1.289, 1.65) {};
  \node[fill=black, circle, inner sep=1pt, label=left:$x^{3}$] at (-0.724,-0.35) {};
  \node[fill=black, circle, inner sep=1pt, label=right:$x^{4}$] at (1.075,-1.159) {};
  \node[fill=black, circle, inner sep=1pt, label=below : $y^{1234}$] at (-0.05,0.1) {};
  \node[fill=black, circle, inner sep=1pt, ] at (0.4, -1.63) {};
  \node[fill=black, circle, inner sep=1pt, ] at (-.7, 2) {};

  \def\aone{0.4}
  \def\bone{-1.63}
  \draw[domain=-4:4.5,smooth,variable=\x,thick] plot ({\x},{-(\x-\aone)^2 + \bone});
  \draw[domain=-4:4.5,smooth,variable=\x,thick] plot ({\x},{(\x-\aone)^2 + \bone});
  \draw[dashed] (-0.05, 0.1) -- (-0.05, 6);
  \fill (\aone,\bone) circle (1pt) node[below] {$y^{134}$};

  \def\atwo{-.7}
  \def\btwo{2}
  \draw[domain=-4.5:4.5,smooth,variable=\x,thick] plot ({\x},{-(\x-\atwo)^2 + \btwo});
  \draw[domain=-4.5:4.5,smooth,variable=\x,thick] plot ({\x},{(\x-\atwo)^2 + \btwo});
  \fill (\atwo,\btwo) circle (1pt) node[below] {$y^{124}$};

\end{tikzpicture}
\caption{
Geometric configuration used in proof of Lemma \ref{dragon}.  The positioning $x^1,x^2\in L^+$, $x^3\in L^-$, and $x^4\in R^-$ forces the placement of $y^{124}$ and $y^{134}$ to agree with this diagram.
}\label{dragon_fig}
\end{figure}
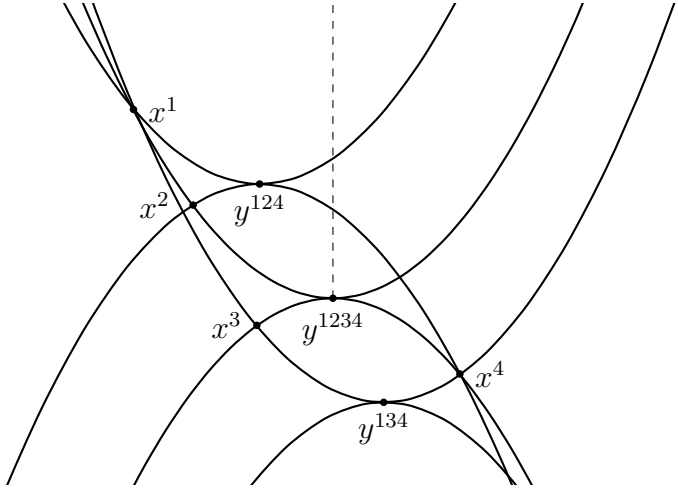
\begin{proof}
We suppose that such a shattering configuration does satisfy 
$$
x^1,x^2\in L^+, \ \ \ x^3\in L^-, \ \ \ x^4\in R^-,
$$
and seek a contradiction.  First, consider $y^{124}$.  Since $x^1,x^2\in P^+$, they cannot both be in $y^{124}+P^+$.  Therefore, $x^2,x^4\in y^{124}+P^-$.  It is easy to see that $y^{124}$ must lie above the parabola $P^+$ and to the left of the vertical axis in order for $y^{124}+P^-$ to meet $x^4\in R^-$.  Since $(y^{124}+P^-)\cap L^+$ in only one point, namely $x^2$, it follows that $x^1\in y^{124}+L^+$.  
\v

\n Next, consider $y^{134}$.  Since $x^3,x^4\in P^-$, they cannot both be in $y^{134}+P^-$, so we conclude that $x^3\in y^{134}+P^+$.  It follows from Lemma \ref{s_unique} that $x^4\in y^{134}+P^+$ as well, since we already have 
$$
\{x^1,x^4\}\subseteq  (-D)\cap (y^{124}-D),
$$
\v

\n and so we cannot have $\{x^1,x^4\}\subseteq y^{134}-D$ as well.  This shows that we cannot have $x^4\in y^{134}+R^-$.  We also cannot have $x^4\in y^{134}+L^-$, since then $y^{134}$ would lie entirely to the right of $P$, and so this contradicts that $x^1\in y^{134}+P^+$.  
\v

\n So far, we have justified that the placement of $y^{124}$ and $y^{134}$ agree with Figure \ref{dragon_fig}.  Finally, we will find a contradiction by considering $y^{234}$.  $y^{234}+P$ is not included in the diagram precisely because there is no geometric possibility.  We see that $x^4\in y^{234}+P^-$, since if $x^4\in y^{234}+P^+$ then so is $x^3$, and we already have
$$
\{x^3,x^4\}\subseteq y^{134}+P^+.
$$
Similarly, we see that $x^3\in y^{234}+P^+$, since we already have
$$
\{x^3,x^4\}\subseteq P^-.
$$
Therefore $x^2\in y^{234}+P^+$ as well. Next we claim that $y^{234}$ lies below and to the right of $y^{134}$.  To see this, note that $y^{134}+P^+$ and $y^{234}+P^+$ are both upward facing parabolas going through the point $x^3$.  Therefore, whichever vertex is farther to the right corresponds to a parabola with a steeper negative slope to the left of $x^3$, and hence lies strictly above the other parabola on the interval between $x^3$ and the point where it intersects the parabola $P^+$.  Therefore, whichever one is farther to the right will also have its corresponding intersection point with $P^+$ be farther to the right.  Thus, we must be describing the parabola going through $x^3$ and $x^2$, not the one going through $x^3$ and $x^1$, because $x^2$ lies to the right of $x^1$.  
\v

\n This immediately gives a contradiction, since we have shown that $y^{234}$ lies below $y^{134}$, and 
$$
x^4\in (y^{134}+P^+)\cap (y^{234}+P^-)
$$
    
\end{proof}

\begin{corollary}
Combining Lemmas~\ref{height_of_points}–\ref{dragon}, we conclude that $x^3$ cannot lie in either branch $P^+$ or $P^-$. Hence, no set of four points in $\mathbb{R}^2$ can be shattered.
\end{corollary}
\v

\n While the intersection properties of parabolas work the same way over any field, many of our arguments in Euclidean space rely in an essential way on one point being farther left than another point, or being above another point.  We are using the fact that $\mathbb{R}$ is an ordered field in an essential way, which does not extend to $\mathbb{F}_q$.  This is not just an inconvenience in the proof mechanism, it actually changes the VC-dimension.  
\v

\subsection{Finite fields}

\n Consider $P=P^+\cup P^-$ over $\mathbb{F}q$, and Let $\mathcal{H}_P(\mathbb{F}_q^2)$ be the hypothesis class consisting of indicator functions $h_y$ of sets of the form $y+P$, $y\in \mathbb{F}_q^2$. Unlike the real case, $\mathbb{F}_q^2$ can shatter 4 points. The table below provides explicit configurations that achieve this for several fields.  
\v
\[
\begin{array}{rcl}
\mathbb{F}_{11}^2 & : & \{x^1=(0,0),\, x^2=(1,2),\, x^3=(2,8),\, x^4=(7,4)\} \\[4pt]
\mathbb{F}_{17}^2 & : & \{x^1=(0,0),\, x^2=(0,1),\, x^3=(1,8),\, x^4=(12,13)\} \\[4pt]
\mathbb{F}_{23}^2 & : & \{x^1=(0,0),\, x^2=(1,2),\, x^3=(10,17),\, x^4=(13,6)\} \\[4pt]
\mathbb{F}_{29}^2 & : & \{x^1=(0,0),\, x^2=(0,2),\, x^3=(8,7),\, x^4=(11,2)\}
\end{array}
\]
\v

\n For the field $\mathbb{F}_{11}$, the centers of the corresponding parabolas are:
\v
\[
\begin{array}{rclcrcl}
y_{1} &=& (0, 0) &\quad& y_{1,2} &=& (9, 4) \\[2pt]
y_{2} &=& (0, 3) &\quad& y_{1,3} &=& (10, 10) \\[2pt]
y_{3} &=& (0, 4) &\quad& y_{1,4} &=& (1, 1) \\[2pt]
y_{4} &=& (0, 9) &\quad& y_{2,3} &=& (0, 1) \\[2pt]
&& &\quad& y_{2,4} &=& (2, 1) \\[2pt]
&& &\quad& y_{3,4} &=& (1, 7) \\[2pt]
&& &\quad& y_{1,2,3} &=& (7, 5) \\[2pt]
&& &\quad& y_{1,2,4} &=& (5, 8) \\[2pt]
&& &\quad& y_{1,3,4} &=& (6, 3) \\[2pt]
&& &\quad& y_{2,3,4} &=& (10, 6) \\[2pt]
&& &\quad& y_{1,2,3,4} &=& (3, 9)
\end{array}
\]
\v

\n We used a computer to find these values, in order to demonstrate the difference between finite fields and the real numbers in this context.  

\newpage

\bibliographystyle{abbrv1}
\bibliography{bibdatabase.bib}

\end{document}